\documentclass[a4paper,notitlepage,twoside,leqno,12pt]{amsart}
\usepackage{graphicx}
\usepackage{float}
\usepackage{bbm,pifont,latexsym}
\usepackage{anysize}

\marginsize{3.8cm}{3.8cm}{3.8cm}{3.6cm}

\usepackage{dcolumn,indentfirst}
\usepackage[colorlinks=false,linkcolor=black,citecolor=black,urlcolor=black,bookmarksopen=true,linktocpage=true,pdfstartview=FitH]{hyperref}
\usepackage{amsmath,amssymb,amscd,amsthm,amsfonts,mathrsfs}
\usepackage{color,graphicx,xcolor,graphics,subfigure,extarrows,caption2}
\usepackage{titletoc}
\newtheorem{thm}{Theorem}[section]

\newtheorem*{cls}{The Closing Lemma}
\newtheorem*{main}{Main Theorem}

\newtheorem{prop}[thm]{Proposition}
\newtheorem{lema}[thm]{Lemma}

\theoremstyle{definition}
\newtheorem{defi}{Definition}

\usepackage{enumerate,enumitem}

\makeatletter\@addtoreset{equation}{section}\makeatother

\begin{document}

\author{Panjing Wu  and  Gaofei Zhang}
\address{Department of Mathematics, Nanjing  University, Nanjing£¬  2100993,   P. R. China}
\email{pgwuxidian@163.com}
\address{Department of Mathematics, QuFu Normal University, Qufu  273165,   P. R. China}
\email{zhanggf@hotmail.com}

\title[]{  Gluing Polynomials along the circle  }

\begin{abstract}Gluing is a cut and paste construction   where the dynamics of a map in a given domain is replaced by a different one, under the condition that the two agree along the gluing curve.    Here  we consider  two polynomials   with  a finite super-attracting fixed point of the same degree. We prove that  any two such non-renormalizable  polynomials can be glued  into a rational map along the Jordan boundary of the immediate basin of the super-attracting fixed point.
\end{abstract}



\date{\today}



\maketitle



\section{Introduction}

Gluing  is a pervasive concept in mathematics which bridges local and global perspectives.  In particular,
Gluing two dynamics into a new one,
 by identifying   two of their subsystems as a single one,  is one of  fundamental ideas in the study of  dynamics. The most well known
  in this regard are mating, intertwining surgery  as well as tuning, see  \cite{BEKM}\cite{DH1}\cite{Do}\cite{E}\cite{EM}\cite{IK}\cite{In}\cite{Re}\cite{ShT1} \cite{SW}\cite{Tan}\cite{YaZa}.  Recognizing
  circle dynamics generated by covering maps as a basic building block in  lower dimensional dynamics, we will
    present a     way in this  work  on how to glue two polynomial dynamics along such a circle.
\subsection{Statement of the main result}

Consider two polynomials $f$ and $g$ with connected Julia sets and of degrees  $d_1$ and $d_2$  respectively, so  that  both $f$ and $g$ have a marked finite super-attracting fixed point of the same degree $d_0$. We further assume that   the $\emph{independence condition}$ holds for both $f$ and $g$, that is,
the other critical orbits do not intersect the  marked immediate super-attracting basins.  By  \cite{RY}   the boundaries of the
immediate super-attracting basins are  Jordan curves on which the two dynamics agree.  Then by gluing $f$ and $g$ along the Jordan boundaries
 of  the marked immediate basins, one can always obtain a topological branched covering $F$ of the sphere to itself  with degree $d_1 + d_2 - d_0$.   We call $F$ the $\emph{topological gluing}$ of $f$ and $g$, see $\S2$ for a precise definition.

 One fascinating nature of holomorphic   dynamics is a deep relationship between the topological  and geometric properties. A cornerstone of this theory is Thurston's theorem for the characterization  of post-critically finite rational maps.  Thus in the  situation here,  a fundamental question is under what conditions, a topological gluing can be realized by a geometric one?

   Let $B_\infty(\cdot)$ and  $B_{0}(\cdot)$  be  the immediate basins at  infinity and the origin respectively. Let   $\mathcal{J}_f$ and $\mathcal{J}_g$  be the Julia sets of $f$ and $g$,  and $\Gamma_f$ and $\Gamma_g$ the Jordan boundaries of the marked super-attracting basins of $f$ and $g$ respectively.  We say a rational map
\begin{equation}\label{fr}
 G(z)  =  \gamma z^m \frac{(z - \alpha^1)\cdots(z-\alpha^k)}{(z - \beta^1)\cdots(z-\beta^l)}
\end{equation} with $m = d_2,\: l = d_2-d_0 ,   \:  k = d_1 - d_0$,   realizes the $\emph{geometric gluing}$ of $f$ and $g$ if  there exist  two  holomorphic   conjugations
  $$\phi:  {B_\infty(f)} \to   {B_\infty(G)}     \: \hbox{and} \:\: \psi:   {B_\infty(g)} \to  {B_0(G)},$$  which   can be   extended  respectively to  conjugations   between the boundaries
$$
 \begin{CD}
         \mathcal{J}_f           @  > \phi      >  >  \partial {B_\infty(G)}        \\
           @V f  VV                         @VV  G  V\\
           \mathcal{J}_f          @  >\phi   >  >       \partial {B_\infty(G)}
     \end{CD}
     $$

     and
     $$
   \begin{CD}
         \mathcal{J}_g           @  > \psi      >  >  \partial {B_0(G)}        \\
           @V g  VV                         @VV  G  V\\
           \mathcal{J}_g          @  >\psi   >  >       \partial {B_0(G)}
     \end{CD}
$$
with  $$\phi(\Gamma_f) = \psi(\Gamma_g) = \partial {B_\infty(G)} \cap \partial {B_0(G)}$$ being a Jordan curve, which is called     the $\emph{gluing curve}$ of $\mathcal{J}_f$ and $\mathcal{J}_g$.

Roughly speaking, a geometric gluing of two polynomials means that,
their  Julia sets    are embedded in   the Julia set of  a rational map $G$  with the dynamics being preserved, and moreover,  the  boundaries  of the two
marked immediate super-attracting basins are glued into a Jordan curve,   so that  the two polynomial Julia sets can be viewed  respectively from the origin and  infinity for $G$.    This property is reminiscent of Bers' celebrated theorem
 on the simultaneous uniformization of Fuchsian groups \cite{BE}, as well as  the work of McMullen on the   simultaneous uniformization of Blaschke products \cite{McM}. For this reason,  we will also  refer to  the geometric gluing in the following   as the $\emph{simultaneous uniformization}$ of two polynomials $f$ and $g$, and denote it as $\mathcal{G}(f, g)$.

   It was known that the geometric gluing always exists for post-critically finite case, which is derived from Thurston's characterization theorem for rational maps \cite{Zhang}.
    The goal of this paper is to   show   that the geometric gluing actually exists   for
     a large class of  post-critically infinite polynomials.   Before we state the main theorem, let us introduce some notions first.

For  $d > d_0 \ge  2$,
let  $\mathcal{P}_{d, d_0}$   denote the space of  polynomials $f$ of degree $d$ with connected Julia sets and
a super-attracting fixed point of degree $d_0$ at the origin,  so that the independence condition holds.  Endow $\mathcal{P}_{d, d_0}$  with the compact open topology.    Let $\mathcal{N}_{d,d_0}$  denote the space of all $f$ in  $\mathcal{P}_{d, d_0}$ which   have  no    indifferent cycles and whose critical points in the Julia set are all non-renormalizable.

Now for  $d_1, d_2  >  d_0 \ge  2$, endow $\mathcal{N}_{d_1, d_0} \times \mathcal{N}_{d_2, d_0}$ with the product topology.
 Let  $m = d_2, l = d_2-d_0$,  and  $k = d_1 - d_0$.  Set
$$
G_{d_1+d_2 - d_0}  = \bigg{\{} \gamma z^m \frac{(z - \alpha^1)\cdots(z-\alpha^k)}{(z - \beta^1)\cdots(z-\beta^l)}\: \bigg{|} \: \alpha^i, \beta^j, \gamma \ne 0  \hbox{ and }  \alpha^i \ne \beta^j \bigg{\}}.
$$
    Let
  $\mathcal{R}_{d_1+ d_2 - d_0}$ be  the quotient space of $G_{d_1+d_2 - d_0}$ modulo scaling conjugations.  Endow $G_{d_1+d_2 - d_0}$ with  the
  compact open topology  and $\mathcal{R}_{d_1+ d_2 - d_0}$   the quotient topology.
   \begin{main}
Let $d_1, d_2 >  d_0 \ge 2$. Then for any pair
$$(f, g) \in \mathcal{N}_{d_1, d_0} \times \mathcal{N}_{d_2, d_0},$$  there is a unique rational map
$G\in \mathcal{R}_{d_1 + d_2 - d_0}$ realizing the  simultaneous uniformization of $f$ and $g$.  Moreover,  the map
$$
\mathcal{G}: \mathcal{N}_{d_1, d_0} \times \mathcal{N}_{d_2, d_0} \to  \mathcal{R}_{d_1 + d_2 - d_0}
$$  by setting  $\mathcal{G}(f, g) = G$ is continuous.
\end{main}

With the aid of quasiconformal surgery, small Julia sets and indifferent cycles for polynomials
 might be exchanged with super-attracting cycles.  In light of these considerations, by the main theorem and surgery technique,
one may expect to   extend $\mathcal{G}$  to the larger space
$\mathcal{P}_{d_1, d_0} \times \mathcal{P}_{d_2, d_0}$.   Among  the challenges  arising  in such an extension is the case  when    a non-locally connected small Julia set for one polynomial is glued to a parabolic periodic point  for the other. In this situation,  it is not even known whether  the geometric gluing exists  that preserves  the topology and dynamics of the small Julia set.  This is closely related to Milnor-Goldberg's parabolic deformation conjecture, which posits  that
one can always replace
a simply connected attracting immediate basin by a parabolic one  while preserving  the dynamics on the Julia set.  For the progress towards the solution of this conjecture, see
\cite{H1}\cite{PR}.

In a manner analogous to
polynomial mating, we will see in the following that  there are parabolic parameters  for which
 the gluing operator $\mathcal{G}$  can be defined yet fails to be continuous.

\subsection{Examples of the gluing}   Consider  the   family
$$
f_\alpha: z \mapsto z^3 + \frac{3}{2}\alpha z^2,\:\: \alpha \in \Bbb C,
$$
whose dynamics and  parameter space  had been extensively investigated by P. Roesch \cite{R1}. Let $\Omega_\alpha$
be the  immediate basin at the origin for $f_\alpha$.   According to \cite{R1} the hyperbolic component $\mathcal{H}_0$ in the parameter space, which contains the origin,
   is a Jordan domain. Moreover,  any $\alpha \in \partial \mathcal{H}_0$ is either  non-renormalizable or a parabolic parameter.  In the previous case,  $\partial \Omega_\alpha$  contains
  the free critical point, and in the later case,
    $\partial \Omega_\alpha$  contains a periodic parabolic orbit. For $\alpha \in \partial \mathcal{H}_0$, the restriction of $f_\alpha$ on $\partial \Omega_\alpha$ is conjugate to the angle doubling map. The angle $\theta$  between
     the fixed point and the free critical point or the parabolic periodic point, to which the free critical point is associated,  is called the $\emph{characteristic angle}$  for $f_\alpha$.  The correspondence between $\theta$ and $\alpha$ is a homeomorphism between $\Bbb T$ and $\partial \mathcal{H}_0$. We denote it as $\alpha = \alpha(\theta)$.
     
      \begin{figure}[!htpb]
\setlength{\unitlength}{1mm}
\begin{center}

\includegraphics[width=70mm]{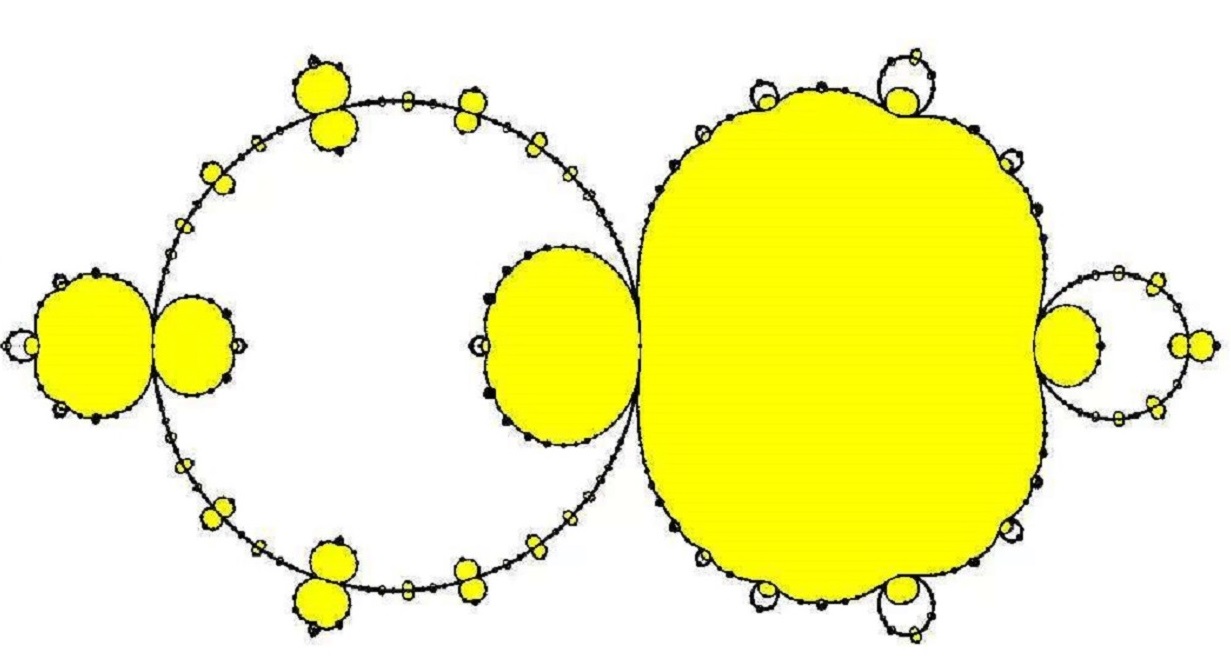}
\caption{The Julia set for $z^3 \frac{z - 3}{1 - 3z}$ with $1$ being a parabolic fixed point having two petals which are symmetric about the unit circle.}
\label{Figure 1}

\end{center}
\end{figure}

 \begin{figure}[!htpb]
\setlength{\unitlength}{1mm}
\begin{center}

\includegraphics[width=60mm]{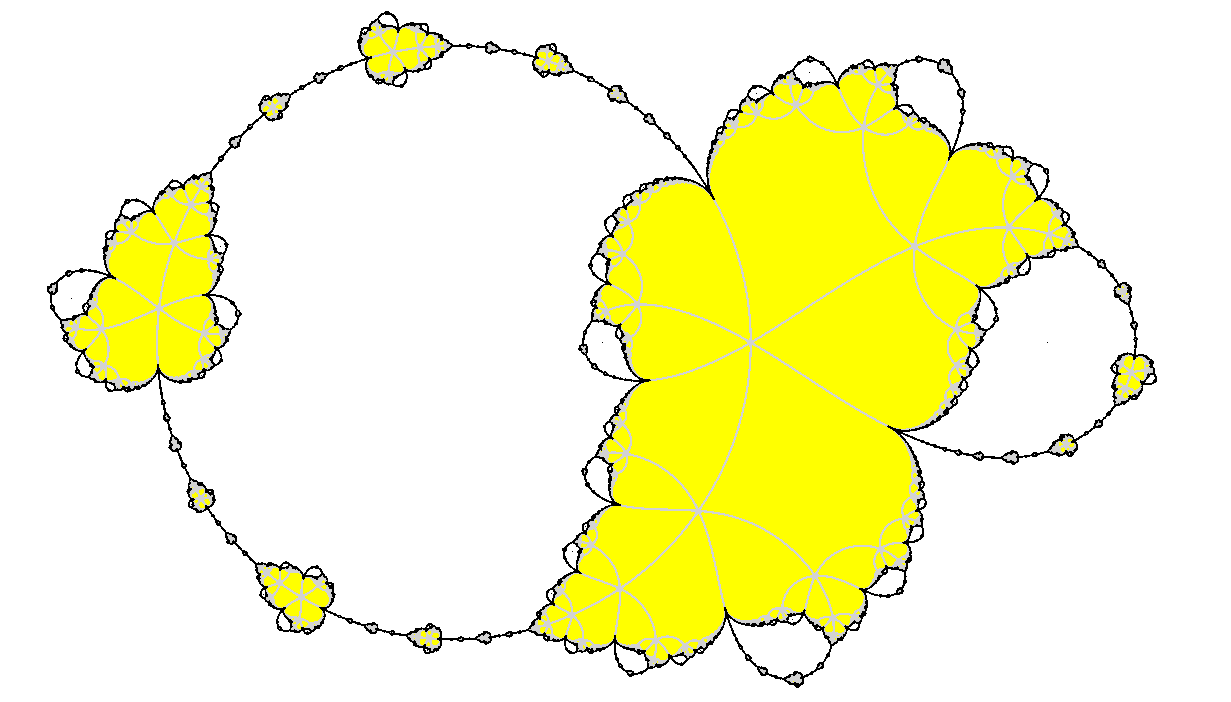}
\caption{The Julia set for $e^{2\pi i \theta_0} z^3 \frac{z - 2}{1 - 2z}$  with   a fixed
 parabolic petal which is symmetric about the unit circle. The  Julia set for $e^{-2\pi i \theta_0} z^3 \frac{z - 2}{1 - 2z}$ is its
  symmetric image about the real line. }
\label{Figure 1}

\end{center}
\end{figure}

 It can be shown that  the geometrical gluing
  can be defined for any two elements in $\partial \mathcal{H}_0$.  By bonding with
  its mirror image,  the map
   $$
   f_{\alpha(\theta)} \to \mathcal{G}(f_{\alpha(\theta)}, f_{\alpha(-\theta)})
   $$   symmetrizes all elements in $\partial \mathcal{H}_0$ to Blaschke products.  It turns out that
    the discontinuity of the    $\mathcal{G}: \partial \mathcal{H}_0 \times \partial \mathcal{H}_0 \to \mathcal{R}_4$ happens exactly at
      those pairs $(f_{\alpha(\theta)}, f_{\alpha(-\theta)})$ for which $\theta$ is periodic under angle doubling map.  For an illustration, let   $\theta = 0$. Then $\alpha(0) = 4i/3$ and $f_{\alpha(0)}$ has a parabolic fixed point in the boundary of the immediate basin at the origin. A simple calculation shows that
$$
\mathcal{G}(f_{\alpha(0)}, f_{\alpha(0)}) = z^3 \frac{z - 3}{1 - 3z}, 
$$
while 
$$
\lim_{\theta \to 0_{+}} \mathcal{G}(f_{\alpha(\theta)}, f_{\alpha(-\theta)}) = e^{2 \pi i \theta_0}z^3 \frac{z - 2}{1 - 2z}  $$ and  $$ \lim_{\theta \to 0_{-}} \mathcal{G}(f_{\alpha(\theta)}, f_{\alpha(-\theta)}) = e^{-2 \pi i \theta_0} z^3 \frac{z - 2}{1 - 2z}
$$ with $ \theta_0 = 0.04892344$.  See Figure 1 and Figure 2 for their Julia sets.

Now let us give an another example to show how two polynomial Julia sets are glued together
along the gluing curve.  Take  $\alpha =   0.72863  + 0.74796 i $
   and $  \beta = 4i/3  $.   The free critical point of $f_\alpha$ belongs to  $\partial \Omega_\alpha$
   and has an infinite forward orbit, and the free critical point of $f_\beta$ belongs to a parabolic basin attached to $\partial \Omega_\beta$. Then
$$
\mathcal{G}(f_\alpha, f_\beta)  = z^3 \frac{z-a}{z-b}
$$
with $a = 1.12078 +1.12078 i$ and  $b = -0.12239 +0.08603i$.

 \vspace{0.5cm}

\begin{figure}[htbp]
 \centering
 \begin{minipage}[t]{0.48\textwidth}
   \includegraphics[width=\textwidth]{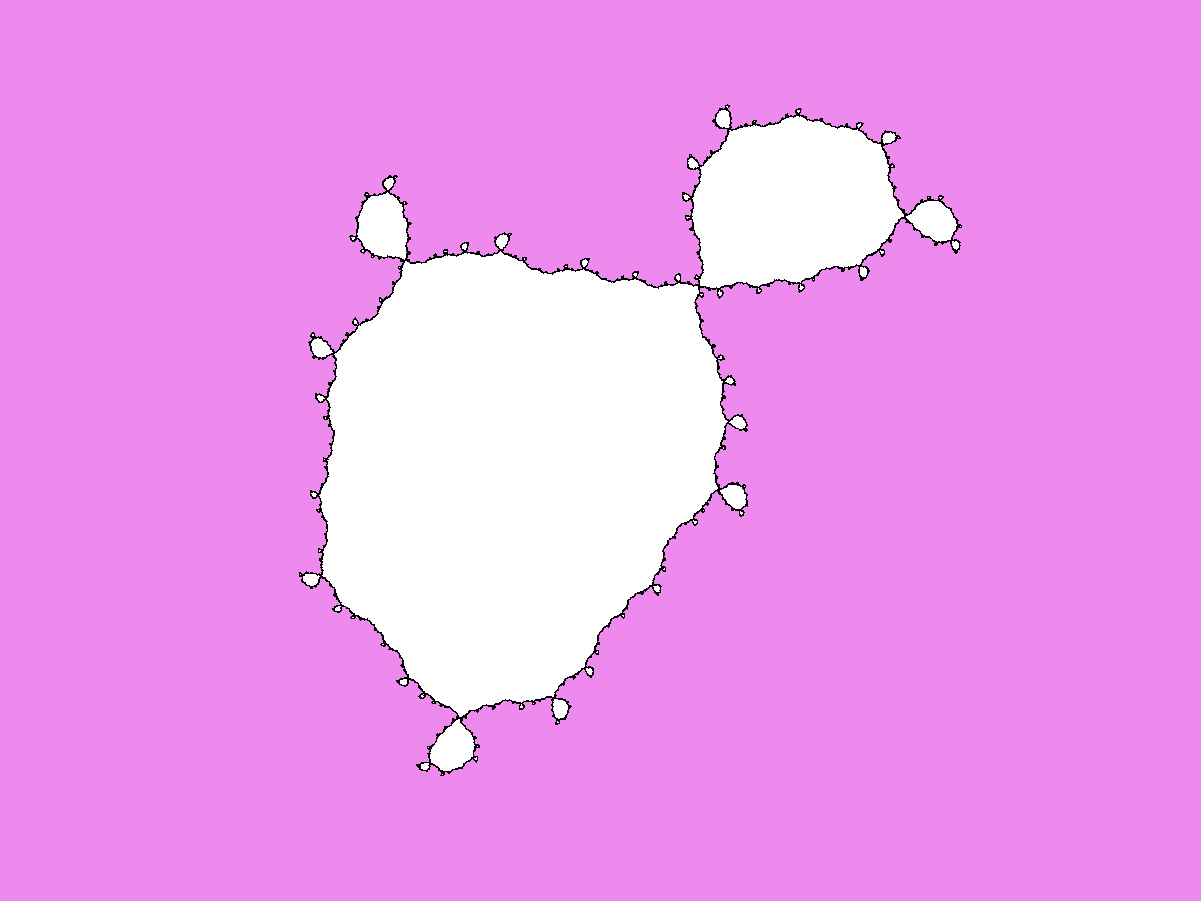}
   \caption{The Julia set for $ f_\alpha $. }
   \label{fig:image1}
 \end{minipage}
 \hfill
 \begin{minipage}[t]{0.48\textwidth}
   \includegraphics[width=\textwidth]{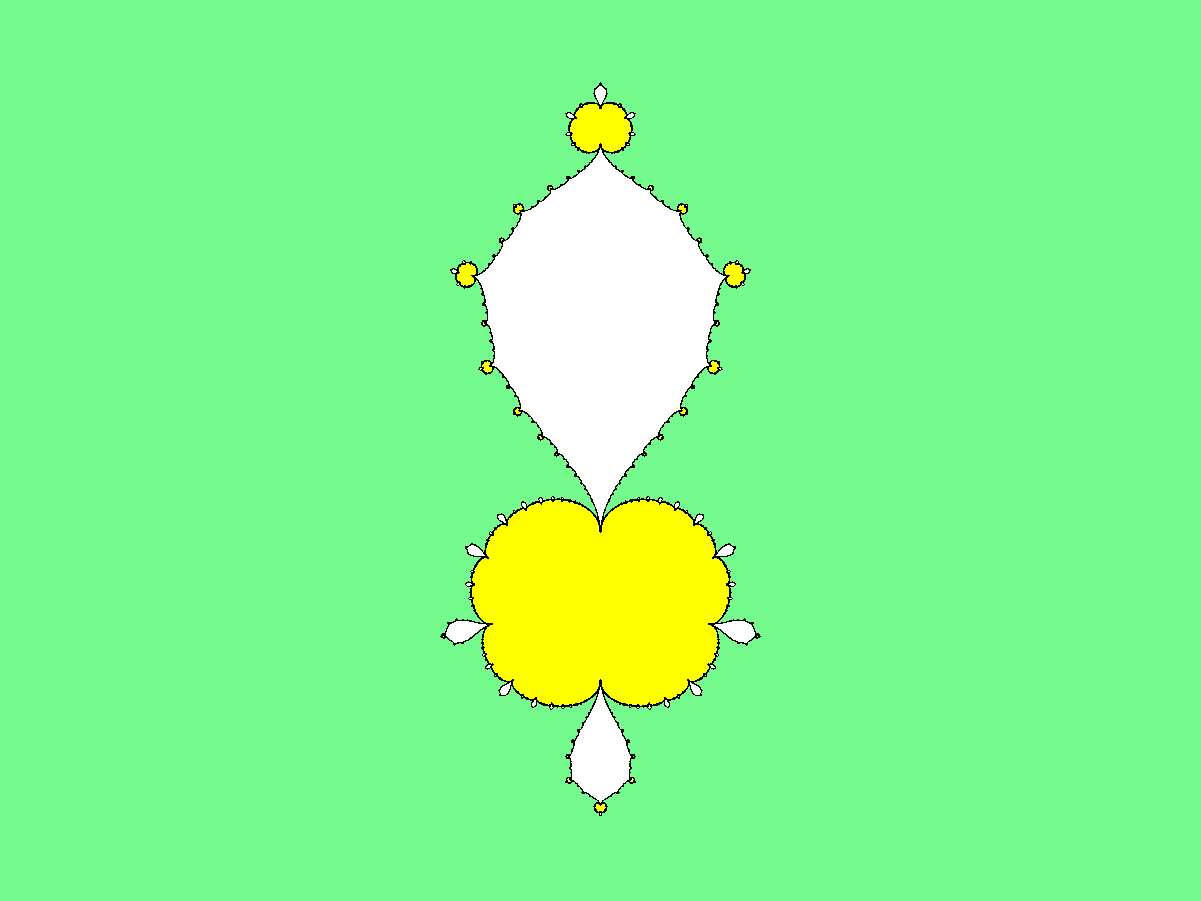}
   \caption{The Julia set for $ f_\beta $. }
   \label{fig:image2}
 \end{minipage}
\end{figure}
 \begin{figure}[!htpb]
\setlength{\unitlength}{1mm}
\begin{center}

\includegraphics[width=70mm]{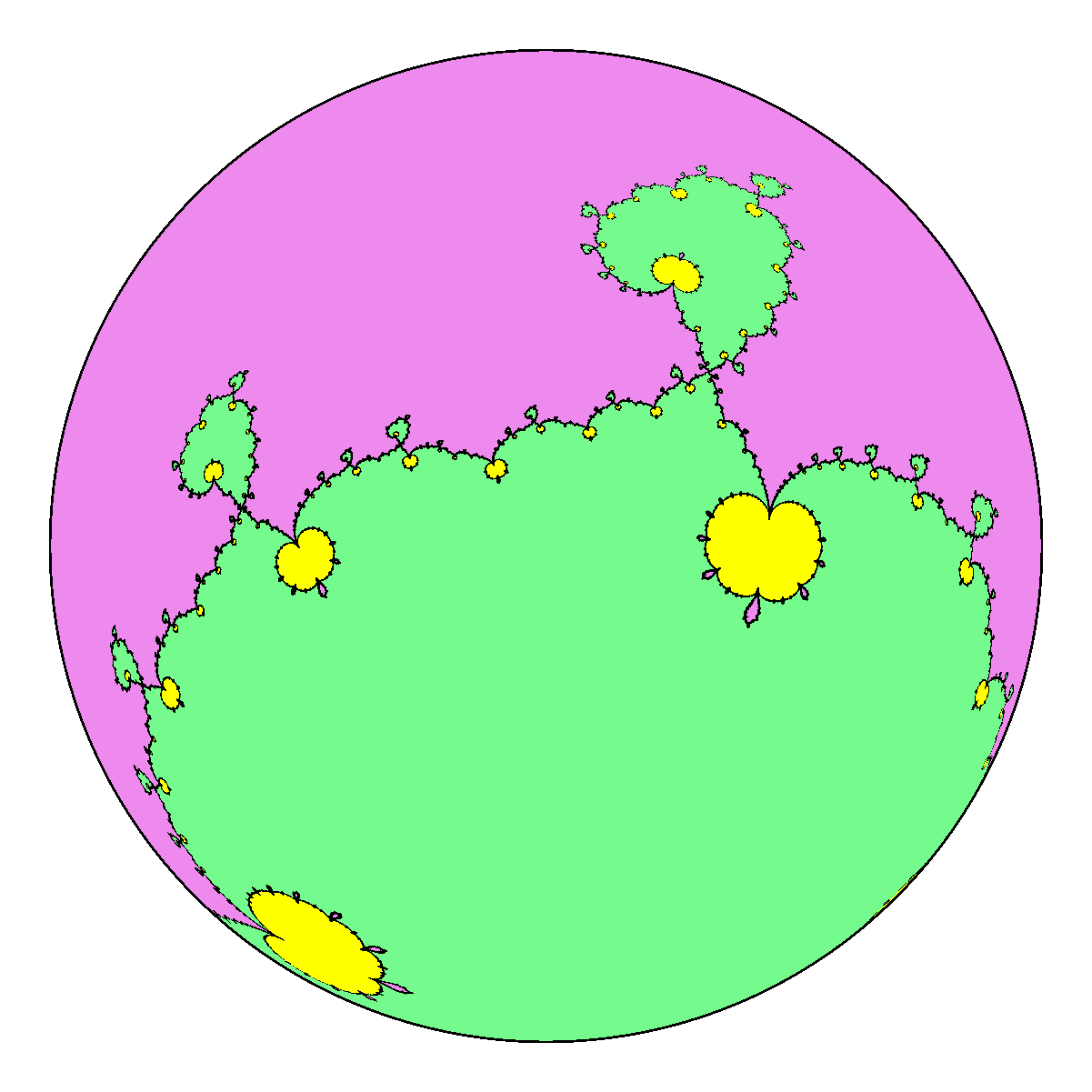}
\caption{The Julia set for $\mathcal{G}(f_\alpha, f_\beta)$. When viewed from  infinity  and the origin, one sees respectively
the Julia sets for $f_\alpha$  and $f_\beta$. }
\label{Figure 1}

\end{center}
\end{figure}

\subsection{Strategy of the proof}

The core of our strategy lies in approximating post-critically infinite maps by post-critically finite ones.
While the general idea of approximation is standard,  it is not always successful,
  largely  due to various  uncontrollable factors that arise  under  the limit process.  To overcome this, we utilize  puzzle structure  and tools from non-renormalizable holomorphic dynamics to ensure that the key  dynamical properties are preserved in the limit.

Concretely,
  we  will  approximate a non-renormalizable polynomial     by  hyperbolic ones with respect to combinatorics.   The idea is as follows.
   Suppose   $(f, g) \in \mathcal{N}_{d_1, d_0} \times \mathcal{N}_{d_2, d_0}$  and   let  $F$ be  the
   topological gluing  of  $ {f}$ and $ {g}$.  Since $F$ behaves like $f$  outside the gluing curve and like $g$  inside, we can choose  an appropriate set of
   equipotential curves and periodic rays in the immediate basins of  infinity and the origin, as well as all other super-attracting cycles. These consist of a connected graph $I$ which yields  a puzzle structure for $F$, as well as   puzzle structures for both $f$ and $g$.

For each $N \ge 1$, we modify $f$  only within the interiors  of  the critical puzzles of depth $N$  so that every  critical point in the Julia set becomes   periodic or eventually enters into some periodic critical orbit. The resulted topological polynomial is thus hyperbolic and is denoted by  $\tilde{f}$. The same process applies to $g$ and the resulted topological polynomial is denoted by $\tilde{g}$.
   By \cite{Zhang}    we can realize  the geometric gluing of $\tilde{f}$ and $\tilde{g}$, denoted as $G_N$. By construction, $G_N$  has the same puzzle structure as $F$ up to  depth  $N$.  In particular,    this argument  provides an alternative    proof that  any non-renormalizable polynomial can be approximated by hyperbolic polynomials,   a result previously established    in \cite{KS} via an analytic approach.

    A key ingredient in the next step   is to show that the sequence  $\{G_N\}$ is compactly contained in  $\mathcal{R}_{d_1 + d_2 - d_0}$.   After that,  we  may  assume that   $G_N \to G$ by taking a subsequence if necessary. It follows that $G$ has the same puzzle structure as $F$. Since $f$ and $g$, and thus $F$ and $G$ are non-renormalizable, as a consequence of non-renormalizable holomorphic dynamics,  points  in $\mathcal{J}_f$ and $\partial B_\infty(G)$ are represented as   nested puzzles whose diameters converge  to zero. The correspondence between the puzzles for $F$ and $G$ yields  a one-to-one correspondence between the points in $\mathcal{J}_f$ and $\partial B_\infty(G)$, which induces a homeomorphism  preserving  the dynamics. The same argument applies to $\mathcal{J}_g$ and $\partial B_0(G)$. It will be seen that the points in $\mathcal{J}_G$,  which can be viewed from both infinity and the origin, form a Jordan curve. This is the   curve along which the two polynomial Julia sets are glued onto  the sphere.

$\bold{Acknowledgement}$.  We would like to thank    Yongcheng Yin   for many inspiring discussions in the early stage of this work.  We further thank Fei Yang,  who provides  all the computer-generated  pictures in this paper.

\tableofcontents

\section{The gluing operator $\mathcal{G}$ for PCF polynomials}
  Suppose $f$ and $g$ are two rational maps of degree $d_1$ and $d_2$,  respectively,  each with  an immediate super-attracting basin, say $D_f$ and $D_g$, whose boundaries are Jordan curves   and which have    the same degree $d_0$. We further assume that the other critical orbits of $f$(resp. $g$)
   do not enter $D_f$ (resp. $D_g$).
Now we can construct a topological map $F: S^2 \to S^2$ by gluing $f$ and $g$ along  $\partial D_f$ and $\partial D_g$.  More precisely, let $\Delta$ be the unit disk and let $\phi: D_f \to \Delta$, $\psi: D_g \to \Delta$ be the holomorphic isomorphisms that  conjugate $f$ and $g$ to the power map $z \to z^{d_0}$. For each $1 \le k \le d_0-1$, there is a homeomorphism

\[\label{m} \Phi = \phi^{-1} \bigg{(}\frac{e^{2 k \pi i/(d_0-1)}} {\psi}\bigg{)}: \partial D_g \to \partial D_f,\]  which reverses the orientation.  We   extend $\Phi$ to a homeomorphism of the sphere so that it maps $\overline { D_g}$ to $ D_f^c$ and maps $D_g^c$ to $\overline {D_f}$.  We then   glue  $ D_f^c$ and $D_g^c$ along their boundaries by identifying $x$ and $\Phi(x)$.  The topological space obtained in this way,
$$
X = D_f^c \bigsqcup_{x \sim \Phi(x)} D_g^c,
$$  is homeomoprhic to $S^2$.   Now we  define a topological map

$$
F: X \to X
$$
by setting
\[\label{oo}
F(z) =
\begin{cases}
f(z) & \text{ for $z \in D_f^c$ and $f(z) \in D_f^c$}, \\
\Phi^{-1}\circ f(z) & \text{ for $z \in D_f^c$ and $f(z) \in D_f$}, \\
g(z) & \text{ for $z \in D_g^c$ and $g(z) \in D_g^c$}, \\
\Phi\circ g(z) & \text{ for $z \in D_g^c$ and $g(z) \in D_g$}.
\end{cases}
\]
We call $F$  the $\emph{topological gluing}$ of $f$ and $g$.  Using the coordinate charts $\overline{D_f}^c$ and ${\overline{D_g}}^c$,  $F$ is holomorphic whenever  a point  is mapped to   the same side of the gluing curve.   In the case that $F$ is post-critically finite and has no Thurston obstructions, by Thurston's  theorem \cite{DH},
there is a rational map $G$ which is combinatorially equivalent to $F$ and thus realizes the simultaneous uniformization of $f$ and $g$.  In this case, we  say $G$ is the $\emph{geometric gluing}$ of $f$ and $g$   and denote
$$
G = \mathcal{G}(f, g).
$$

  It was proved by Zhang that, for    post-critically finite polynomials $f$ and $g$, the topological map $F$ obtained in this way has no Thurston obstructions.  Therefore, $f$ and $g$
  can be always  glued into a rational map $G$.
\begin{thm}[Zhang, \cite{Zhang}]\label{THZ}
Let  $f \in\mathcal{P}_{d_1, d_0}$ and $g \in\mathcal{P}_{d_2, d_0} $  be two post-critically finite polynomials satisfying the independence condition. Then they can be glued  into a rational map $G \in \mathcal{R}_{d_1 + d_2 - d_0}$.
\end{thm}

In the following we need a slightly more general version of the theorem that allows for attracting cycles. In this case,  one may use a surgery to convert  the cycle to a super-attracting cycle, and revert the change by surgery after the gluing.  Thus,  in what follows  we   assume that all the attracting cycles are super-attracting.

\section{Puzzle structure for the topological gluing}

\subsection{Puzzle construction for polynomials}
Before   constructing the puzzle system for the topological gluing, let us give a brief  overview of the construction of a puzzle system for a polynomial. We refer the reader to \cite{B}\cite{BH}\cite{H}\cite{Mil}\cite{Yo} for a detailed definition.
 Suppose $f$ is a polynomial of degree $d \ge 2$ and with no indifferent cycles.  Let $I$ be the graph consisting of an appropriate set  of   periodic external rays  and an equipotential curve,
 as well as periodic  internal rays and inner equipotential curves for attracting  basins. Let $I_n$ be the pull back of $I$ by $f^n$ for $n \ge 1$. Each bounded component of $\Bbb C - I_n$ is called a puzzle of depth $n$.   For $z \in \mathcal{J}_f$, let $P_n(z)$ denote the puzzle of depth $n$ which contains $z$. Then any two puzzles are either disjoint or nested. Moreover,
 $P_{n+1}(z) \subset P_n(z)$ and $f(P_{n+1}(z)) = P_n(f(z))$ for all $n \ge 0$.

We say a critical point $c$ is  $\emph {renormalizable}$ with respect to $\mathcal{P}$ if   there exists  an integer $p \ge 1$,  such that
 $$
 P_{n}(c)  = P_{n}(f^{kp}(c))
 $$  for all $k,  n \ge 0$.    Otherwise we say $c$ is $\emph{non-renormalizable}$ with respect to $\mathcal{P}$. We say $c$ is $\emph{non-renormalizable}$ if it is  non-renormalizable with respect to any puzzle system  $\mathcal{P}$.
Recall  that $\mathcal{N}_{d, d_0}$ denotes  the space  of polynomials in $\mathcal{P}_{d, d_0}$  with  no indifferent cycles and whose
 critical points in the Julia set are all  non-renormalizable.

\subsection{Puzzle construction for topological gluing}
Let $f \in \mathcal{N}_{d_1, d_0}$ and $g \in \mathcal{N}_{d_2, d_0}$.
Let $F$ be the topological  gluing of $f$ and $g$ as described in $\S2$. Then $F$ has an invariant curve $\Gamma$ so that $F: \Gamma \to \Gamma$ is conjugate to $z \mapsto z^{d_0}$.   Moreover,  it behaves like $f$  outside    $\Gamma$
and like $g$   inside.  So we can speak of rays and equipotential curves for $F$.  Take a set of  periodic orbits $\{z_1, \cdots, z_p\}$ in $\Gamma$ such that for each $z_i$
there are exactly two rays of $F$ landing at $z_i$, one starting from  infinity and the other  from the origin.    Let $R_i$ denote the union of the two rays. Let $E_\infty$ and $E_0$ be two equipotential curves of $F$  that belong to the outside and inside of $\Gamma$,  respectively.  For attracting cycles of $f$ and $g$ that  may exist, as in \cite{RY}, we  may take appropriate sets of
internal rays,  inner equipotential curves,  and  external rays, say $\mathcal{R}_{int}$, $\mathcal{E}_{int}$ and
 $\mathcal{R}_{ext}$ , so that
\begin{equation} \label{RE}
I = \{R_i, 1 \le i \le p, E_\infty, E_0\} \cup \mathcal{R}_{int} \cup \mathcal{E}_{int} \cup \mathcal{R}_{ext}
\end{equation}
is an admissible initial     graph of $F$. We may assume that the rays in $I$ do not land on the post-critical points.

Let $\mathcal{C}$ denote the set of all points in super-attracting cycles. In particular, $\{0, \infty\} \subset \mathcal{C}$.

\begin{defi}
For $n\ge 0$, let $I_n = F^{-n}(I)$. The components
of $\widehat{\Bbb C} - I_n$,  that  do not contain any point in $F^{-n}(\mathcal{C})$,  are called puzzles of depth $n$. We denote by $P_n(z)$  the puzzle of depth $n$ that contains $z$.
\end{defi}

 Let $\mathcal{P}$ denote the puzzle system induced by $I$. By definition we readily obtain
\begin{lema}\label{pp}
The following properties hold.
\begin{itemize}
\item $P_{n+1}(z) \subset P_{n}(z)$ and $F(P_{n+1}(z))  =  P_{n}(F(z))$  for all $z$ and  $n \ge 0$;
\item any two puzzles in $\mathcal{P}$ are either nested or disjoint;
\item $I_n$ is connected  for all $n \ge 0$, hence all puzzles in $\mathcal{P}$ are simply connected.
\end{itemize}
\end{lema}

   The portion of the initial graph $I$ lying  outside  the gluing curve  induces an
   initial graph  for $f$,  while  the portion   inside induces that for $g$. These initial graphs yield  puzzle systems for $f$ and $g$,  respectively. Let $\mathcal{P}_F$, $\mathcal{P}_f$, and $\mathcal{P}_g$ denote the puzzle systems for $F$, $f$ and $g$, respectively.   Because  $f$ and $g$ satisfy the independence condition,
   puzzles in $\mathcal{P}_F$   that  contain  critical orbits for $F$ correspond to
    puzzles  in  $\mathcal{P}_f$ and $\mathcal{P}_g$ that  contain  critical orbits for $f$ and $g$.  This correspondence  preserves the dynamics.    Since $f \in \mathcal{N}_{d_1, d_0}$ and $g \in \mathcal{N}_{d_2, d_0}$, all the critical points in the ``Julia set'' of $F$ , i.e., the critical points which  do not belong to the super-attracting cycles,  are non-renormalizable.

\vspace{0.2cm}
\noindent
$\bold{A\: toy \:model\: for\: the \:puzzle\: construction}$  Let us illustrate
  this puzzle construction for the pair of $f_a$ and $f_b$ in the family $f_\alpha: z \mapsto z^3 + \frac{3}{2}\alpha z^2$,   with $a$ and $b$ belonging to  $\partial \mathcal{H}_0$  and having  characteristic angles
$7 \pi/6$ and $11 \pi/6$,  respectively.  That is, the forward orbits of the two  critical points  on the gluing curve are given by the angle doubling map
$$
7 \pi/6 \to   \pi/3 \to 2 \pi/3 \to 4 \pi/3 \to 2 \pi/3,
$$  and
$$
11 \pi/6 \to  5 \pi/3 \to  4 \pi/3 \to 2 \pi/3 \to 2 \pi/3.
$$

 The initial  graph $I$ is the union of two equipotential curves, one lying in the neighborhood of  infinity and the other  in the neighborhood of the origin, together with  two fixed rays (one from  infinity  and one  from the origin) that     land on the unique fixed point on the gluing curve.  Thus  there is  only one puzzle of depth $0$, namely $P_{0,1}$.  The pull back of $P_{0,1}$ yields
two puzzles of depth $1$, denoted $P_{1,i}$ for  $i = 1, 2$,  whose  pull back  yields
 six puzzles of depth $2$, denoted $P_{2,i}$ for $1 \le i \le 6$.  The  orbit relations among these puzzles are
$$
P_{2,1} \to P_{1,2} \to P_{0,1},\:\: P_{2,6} \to P_{1,2}\to P_{0,1}, \:\: P_{2,3} \to P_{1,2} \to P_{0,1}, $$$$ P_{2,4} \to P_{1,1} \to P_{0,1},\:\: P_{2,5} \to P_{1,1} \to P_{0,1}, \:\: P_{2,2} \to P_{1,1} \to P_{0,1}.
$$
In  this model,  $\mathcal{C} = \{0, \infty\}$,  and  the white region consists of those  components that  contain  $\infty$, the origin and their pre-images.

\begin{figure}[!htpb]
\setlength{\unitlength}{1mm}
\begin{center}

\includegraphics[width=80mm]{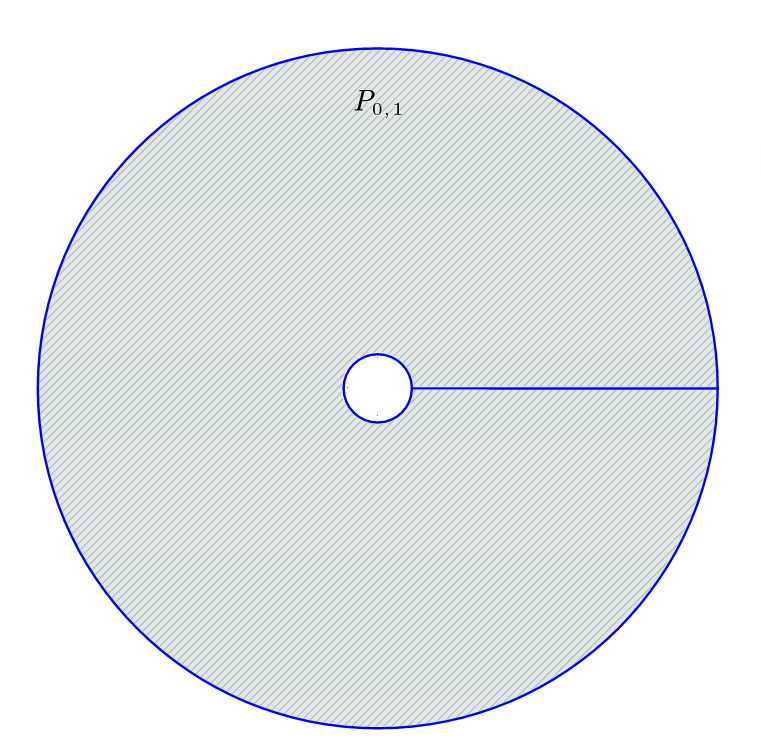}
\caption{Puzzles  of  depth $0$.}
\label{Figure 1}

\end{center}
\end{figure}

\begin{figure}[!htpb]
\setlength{\unitlength}{1mm}
\begin{center}

\includegraphics[width=80mm]{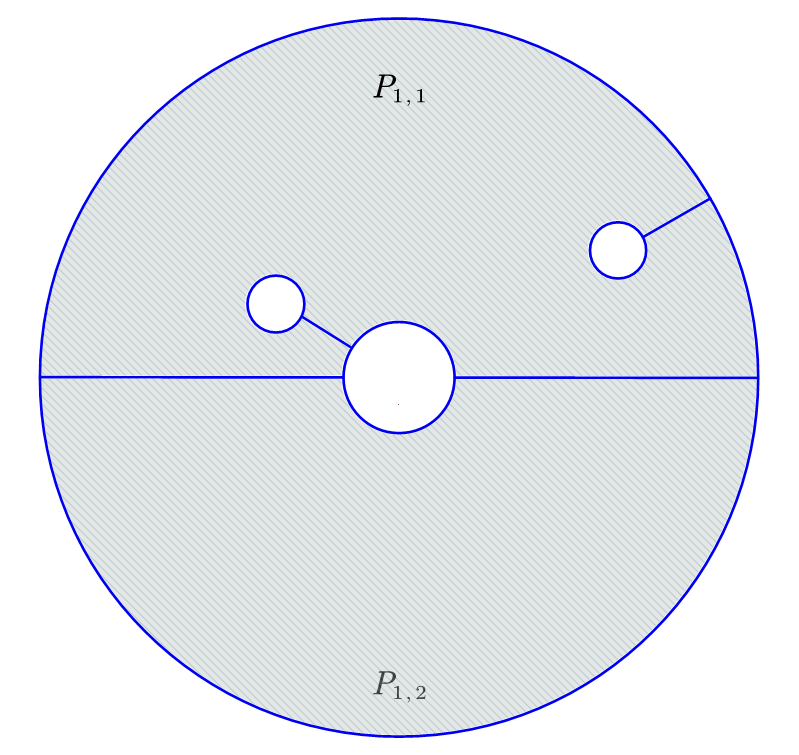}
\caption{Puzzles  of  depth $1$.}
\label{Figure 1}

\end{center}
\end{figure}

\begin{figure}[!htpb]
\setlength{\unitlength}{1mm}
\begin{center}

\includegraphics[width=80mm]{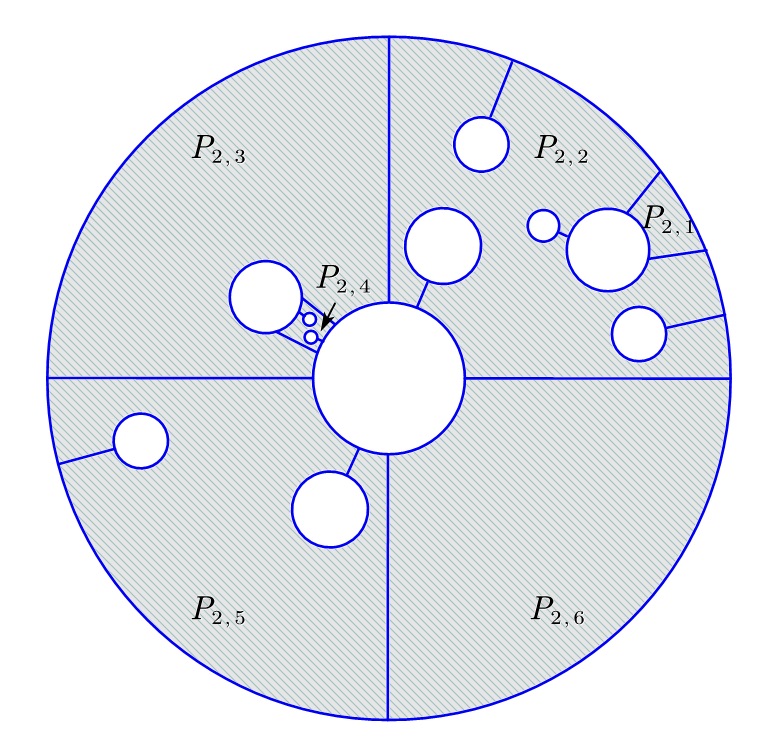}
\caption{Puzzles  of  depth $2$.}
\label{Figure 1}

\end{center}
\end{figure}

Puzzle systems  are usually established for holomorphic maps to study the topology and the geometry of the Julia sets, see \cite{B}\cite{BH}\cite{H}\cite{L}\cite{Mil}\cite{P1}\cite{P2}\cite{PZ} \cite{Yo}   for polynomials and \cite{CDK}\cite{D1}\cite{D2}\cite{QWY}\cite{R3}\cite{R2}\cite{WYZ}   for rational maps. This is by no means an exhaustive  list.
The puzzle construction here is purely topological.  We use it to record  the topological data that characterize the rational map realizing  the geometric gluing.  However,  typically  the topological gluing here is post-critically infinite, a fact lies  beyond the scope of Thurston's characterization theorem.    From the point of this view, this work is an initial effort to extend  Thurston's characterization theorem to non-renormalizable rational maps via  puzzle structures.

 Now let us return  to the real word to see how the puzzle structure for $F$ is used in our proof. In $\S4$ we   perturb $f$ and $g$ to obtain  hyperbolic topological polynomials  $f_n$ and $g_n$ in the sense of combinatorics (here hyperbolic means that all  critical points are periodic or enter into some critical periodic orbit). More precisely, we choose $f_n$ and $g_n$ so that the topological gluing of $f_n$ and $g_n$, denoted  $F_n$, has the same puzzle structure as  $F$ up to  depth $n$. Since $f_n$ and $g_n$ are hyperbolic, they are combinatorially equivalent to polynomials by Thurston's characterization theorem.  By Theorem~\ref{THZ}, for all $n \ge 1$,  we obtain   the geometric gluing of $f_n$ and $g_n$,
$$
G_n = \mathcal{G}(f_{n}, g_{n}).
$$
It  is clear that  for any $N\ge 1$,
$G_n$ and $F$ have the same puzzle structure  up to depth $N$ for all $n \ge N$.  In $\S5$ we  prove that $\{G_n\}$ lies in a compact family. By taking a subsequence if necessary,  we obtain  a limit map $G$.  By selecting  sufficiently many repelling periodic rays in the definition of the initial graph $I$, and disregarding some of them if necessary, we can assume that the rays in $I$  continue to land on repelling periodic orbits after taking the limit.      It follows  that $G$ and $F$ share  the same puzzle structure induced by $I$.   With the aid of the tools developed in non-renormalizable holomorphic dynamics, we can  show that  $G$ realizes the geometric gluing of $f$ and $g$.

\section{The closing Lemma}

The idea in this section is of independent interest. We shall see that it  can be used to give an alternative   proof that
  a non-renormalizable polynomial can be approximated by critically finite hyperbolic polynomials.
This was proved in \cite{KS} via an analytic approach. The goal here  is to establish the following

\begin{cls}
Let $f \in \mathcal{N}_{d, d_0}$.  Then for any $N \ge 1$,  one can  obtain a hyperbolic topological polynomial by perturbing $f$ only in critical puzzles of depth $N$,   in such a way  that every  critical point in the Julia set either becomes periodic or enters a  periodic cycle that
 contain  critical points.
\end{cls}

The basic idea  of the  proof
 is to close the orbit of  a  candidate point near the critical point.       As noted in $\S3$,  the initial graph $I$ for $F$ induces puzzle systems for both $f$ and $g$.    More precisely,  take
   the  part of $I$ in (\ref{RE})  lying outside     the gluing curve, and replace the part of $I$   inside    of the gluing curve by  an inner equipotential curve together with a group of corresponding internal rays.  The new graph is an initial
  graph for  $f$  and  generates a puzzle structure for $f$.
  Choose  an $N \ge  1$ large enough so that every  puzzle of depth $N$  contains   at most one critical point in the Julia set.  Here we use the assumption that $f$ is non-renormalizable with respect to this puzzle system.    Consequently,  any nest of puzzles shrinks to a single point\cite{KL1}\cite{KL2}\cite{KSS}\cite{KS}\cite{QY}.
\begin{lema}\label{cand}
Suppose $c$ is a critical point in the Julia set of $f$ and $N \ge 1$. Then there exist a $z_0 \in P_N(c)$ and a $k \ge 1$ such that the puzzles
$P_N(f^l(z_0))$ for  $0 \le l \le k-1$  are pairwise  disjoint   and $P_N(f^k(z_0)) = P_N(c)$.
\end{lema}
\begin{proof} By the assumption that the rays in $I$ do not land on any post-critical point,  $c$ belongs to the interior of $P_N(c)$.
Since $c \in \mathcal{J}_f$, $P_N(c)$ is an open neighborhood of a Julia point.    By the eventually-onto property of the Julia set, there exists  a $k \ge 1$ with
$$
f^k(P_N(c) \cap \mathcal{J}_f) = \mathcal{J}_f.
$$
Hence  there are points in  $P_N(c) \cap \mathcal{J}_f$  whose forward orbit
 returns  to  $P_N(c)$. Choose  $z_0$  among  such  points so that the return time is minimal.
 The disjointness of the puzzles $P_N(f^l(z_0))$ for $0 \le l \le k-1$  then follows from  this minimality.

\end{proof}

\subsection{Proof of the Closing Lemma}

We  will modify $f$ into a topological polynomial so that
each critical point in the Julia set  either becomes periodic or enters  a periodic cycle containing   critical points,
   and each critical orbit meets  a   depth-$N$  puzzle at most once.

Beginning  with an arbitrary critical point
 $c$ in the Julia set,  let   $$ \mathcal{O}_c = \{P_N^0, \cdots, P_N^{k-1}\}$$
be   the finite set of depth-$N$  puzzles    and  $z_0$    the candidate point
   guaranteed  by Lemma~\ref{cand},  with  $z_l  = f^l(z_0) \in P_N^l$ for $0 \le l \le k-1$ and $z_k = f^{k}(z_0) \in P_N^0$.     We will  modify
$f$ only
  inside   $P_N^0$  to obtain a   topological polynomial $\tilde{f}$  with the following properties:
\begin{itemize}
\item $\tilde{f}$ has a critical point at $z_k$ with local degree $\deg \tilde{f}|_{z_k} = \deg f|_{c}$,
\item $\tilde{f}(z_k) = z_1$.
\end{itemize} This is always possible.  Indeed,  choose a Jordan curve $\gamma$
in $P_{N-1}(z_1)$ that  surrounds $v = f(c)$, $z_1$ and $z_{k+1}$. Since $c$ is the unique critical point in $P_N^0 = P_{N}(z_0)$, there is a unique component $\xi$ of $f^{-1}(\gamma)$ contained in $P_{N}(z_0)$ whose covering degree
$f: \xi \to \gamma$   equals   the local degree of $f$ at $c$. We may  modify $f$ only  inside $\xi$ so that $z_{k}$ becomes the critical point of   $\tilde{f}$ and  maps to $z_1$.    The same idea extends  to all the  subsequent modifications (see Figure 9  for   illustration).

\begin{figure}[!htpb]
\setlength{\unitlength}{1mm}
\begin{center}

\includegraphics[width=120mm]{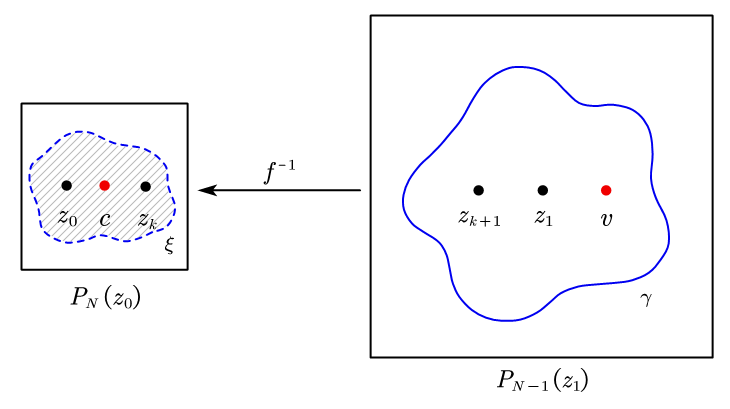}
\caption{Modifying $f$ only in the interior of $P_N(c)$.}
\label{Figure 1}

\end{center}
\end{figure}
Now if for some $1 \le l \le k-1$,
$P_N^l$ contains a critical point, say $c'$, we further  modify $\tilde{f}$ only inside  the  interior of
$P_N^l$ to  obtain a new topological polynomial (still denoted  $\tilde{f}$)  such that
\begin{itemize}
\item $\tilde{f}$ has a critical point at $z_l$ with $\deg \tilde{f}|_{z_l} = \deg \tilde{f}|_{c'}$, and
\item $\tilde{f}(z_{l}) = z_{l+1}$.
\end{itemize}
In this way we can  ensure that  all the critical points in the puzzles of $\mathcal{O}_c$  lie in  a single periodic orbit   $\{z_l, 1 \le l \le k\}$.

\begin{figure}[H]
\setlength{\unitlength}{1mm}
\begin{center}

\includegraphics[width=115mm]{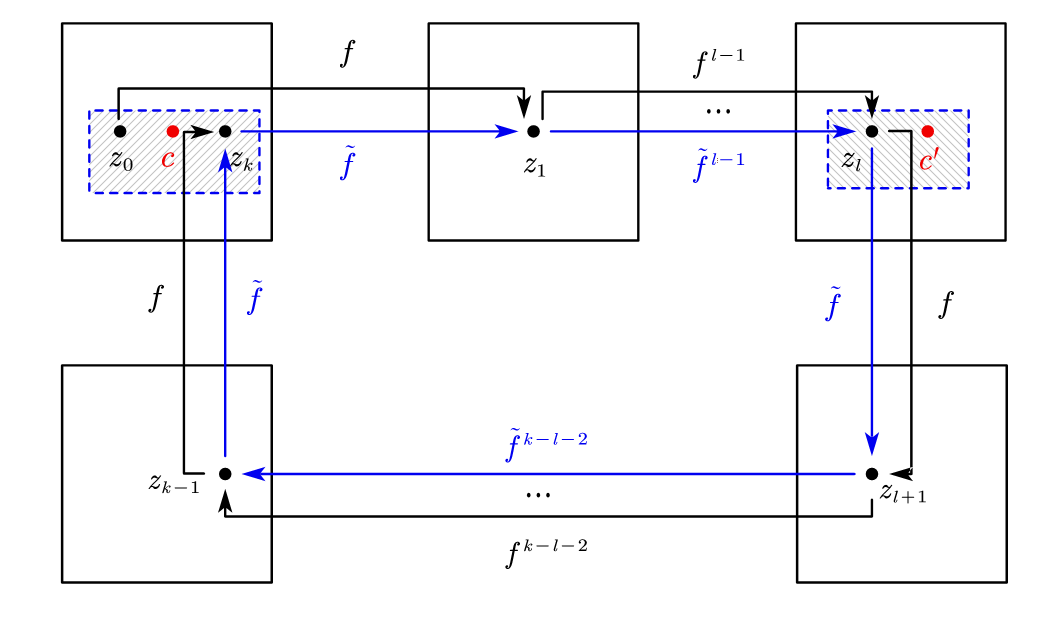}
\caption{Closing the orbit of the candidate point $z_0$.}
\label{Figure 1}

\end{center}
\end{figure}

By induction,  suppose we have a finite set of puzzles of depth-$N$ and a finite set of points,  $$\mathcal{T} = \{P_N^1, \cdots, P_N^n\}  \:\: \hbox{and}\:\: \mathcal{S} = \{z_1, \cdots, z_n\}$$ such that
\begin{itemize}
\item $z_i \in P_N^i$ for all $1 \le i \le n$;
\item for $1 \le i \le n$, $\tilde{f}(z_i) \in \mathcal{S}$;
\item if $P_N^i$ contains a critical point $c$ for some $i$, then $z_i = c$;
\item the forward orbit of each $z_i$ contains at least one critical point;
\item  $f$ and $\tilde{f}$ coincide outside the interiors of the critical puzzles in $\mathcal{T}$.
\end{itemize}

If  all the critical points of $\tilde{f}$ are periodic or eventually enter   some critical periodic orbit, we have done. Otherwise,
there is a  critical point $\tilde{c}$ in the Julia set of $f$  that  does  not belong to the puzzles in $\mathcal{T}$.   Let $\mathcal{O}_{\tilde{c}}$ be the finite set of puzzles and $\tilde{z}_0$  the candidate point  guaranteed by Lemma~\ref{cand}.

In the case that $\mathcal{O}_{\tilde{c}} \cap \mathcal{T} = \emptyset$, apply the previous
 argument for $\mathcal{O}_c$ to $\mathcal{O}_{\tilde{c}}$ and  obtain  a single critical periodic orbit contained in the puzzles of $\mathcal{O}_{\tilde{c}}$.
  Then  add $\mathcal{O}_{\tilde{c}}$ to $\mathcal{T}$ and add the corresponding critical   orbit to  $\mathcal{S}$.

In the case $\mathcal{O}_{\tilde{c}} \cap \mathcal{T} \ne \emptyset$, let
    $\tilde{P}_N^l$  be  the first puzzle in  $\mathcal{O}_{\tilde{c}}$ that  belongs to $\mathcal{T}$, i.e.,
   $$
   \tilde{P}_N^0, \cdots, \tilde{P}_N^l = P_N^m \in \mathcal{T}.
   $$   We may modify $\tilde{f}$ only inside  the interior
    of $\tilde{P}_N^0$
    so that $\tilde{c}$  remains  a critical point of $\tilde{f}$  with the same degree and $\tilde{f}^l(\tilde{c}) = z_m$.  In fact, since $\tilde{f}^l(\tilde{z}_0) \in P_N^m(z_m)$, we can choose   $w \in P_{N+l-1}(\tilde{f}(\tilde{z}_0))$ so that $\tilde{f}^{l-1}(w) = z_m$.  As before,  take a Jordan curve $\gamma$
    in $P_{N-1}(\tilde{f}(\tilde{z}_0))$   surrounding  $\tilde{f}(\tilde{z}_0)$,  $\tilde{v} = \tilde{f}(\tilde{c})$ and $w$.  Let $\xi$ be the component of $\tilde{f}^{-1}(\gamma)$ in $\tilde{P}_N^0 = P_N(\tilde{z}_0)$ that surrounds $\tilde{z}_0$, $\tilde{c}$ and $\tilde{w}$ with $\tilde{f}(\tilde{w}) = w$.
    We can modify $\tilde{f}$ only in the interior of $\xi$  so that $\tilde{c}$ remains a critical point of
    $\tilde{f}$ with the same degree, and moreover, $\tilde{f}(\tilde{c}) = w$.  Thus $\tilde{c}$ is eventually mapped to $z_m$ through
    $$
    \tilde{c},  w = \tilde{f}(\tilde{c}),  \cdots,   \tilde{f}^{l-1}(w)  = \tilde{f}^l(\tilde{c})   = z_m.
    $$
    
     \begin{figure}[H]
\setlength{\unitlength}{1.2mm}
\begin{center}

\includegraphics[width=125mm]{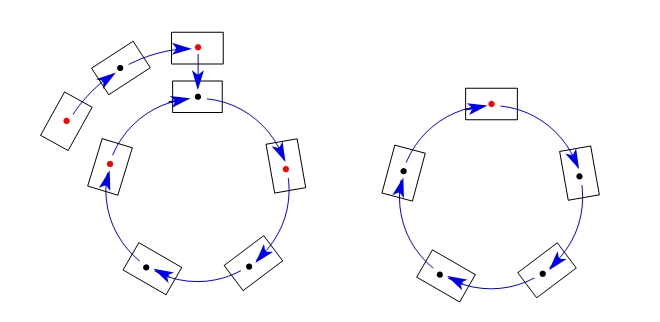}
\caption{The critical orbits for $\tilde{f}$. }
\label{Figure 1}

\end{center}
\end{figure}
For   $1 \le i \le l-1$, if some
    $\tilde{P}_N^i$ contains a critical point $c$, we may modify $\tilde{f}$ only inside  the interior
     of $\tilde{P}_N^i$ so that $\tilde{f}^i(\tilde{c})$  becomes the  critical point of $\tilde{f}$ in $\tilde{P}_N^i$ with the same degree  as $c$, and  is stilled mapped to $\tilde{f}^{i+1}(\tilde{c})$.    After this, we add  $\tilde{P}_N^0, \cdots, \tilde{P}_N^{l-1}$  to $\mathcal{T}$ and $\tilde{f}^i(\tilde{c}), 1 \le i \le l-1$, to  $\mathcal{S}$.

Since there are only finitely many critical points in $\mathcal{J}_f$, the  argument terminates  after finitely many steps.

\subsection{Handling the cluing curve}

 Note that a problem arises when we   modify $f$ in a puzzle   which intersects  the gluing curve. Such a modification may change the dynamics on the gluing curve, where the two dynamics   are supposed to coincide, and  the two new maps can no longer  be glued along the curve.

    To overcome this,  we   push the gluing curve inward  before  making  any  modification.   Concretely,
     suppose  the unit disk is the marked immediate  super-attracting basin and  $f$ is   $z \mapsto z^{d_0}$ there,  and    $|z| = r$ is the   inner equipotential curve used to construct the  depth-$N$ puzzles.  Take $r < \tilde{r} < 1$. We first push
     the gluing curve for $f$ to the circle $\{z: \: |z| = \tilde{r}\}$. To do this,
      modify $f$  in the annulus $r < |z| < 1$ so that
 \begin{itemize}
 \item  $f(t e^{i\theta})=  t e^{i d_0 \theta}$ for $ \tilde{r} \le  t \le 1$,
 \item   $f(t e^{i\theta})=  \eta(t)  e^{i d_0 \theta}$ for $ r  \le  t \le \tilde{r}   $,  where $\eta$ is a linear with  $\eta(r) = r^{d_0}$ and $\eta(\tilde{r}) = \tilde{r}$.
 \end{itemize}

  Now,   when we modify $f$ in a puzzle  that intersects the gluing curve,
 we can perform  the modification  only   outside   $\{z: \: |z| = \tilde{r}\}$. In fact,  since
 the points $z_{k+1}$, $z_1$ and $v$ in the right hand of Figure 9
  lie in  the Julia set,  the  Jordan curve $\gamma$ there  can be taken  outside   $\{z:\: |z| = \tilde{r}\}$,
    and thus the curve
  $\xi$
  also  lies    outside.    In particular, the
 circle $\{z:\: |z| = \tilde{r}\}$ is $\tilde{f}$-invariant  and  will serve as  the gluing curve for $\tilde{f}$.

$\emph{Note}$.  It is interesting  to consider the case when there are critical points lying on the gluing curve.
In this case, the polynomial is perturbed so that the critical points  either
lie on or enter into
  to a long periodic critical orbit near the gluing curve. We shall see,  in the limit,   that since
  the puzzle containing the critical point intersects the gluing curve and since the size of the puzzle tends  to zero, the critical point will be pushed back to the correct  position on  the gluing curve.

\subsection{Post-critically finite approximation}
Now apply the above process to  $g$ and obtain $\tilde{g}$.  Since both
 $\tilde{f}$ and $\tilde{g}$ are hyperbolic  topological polynomials, i.e. their
  critical points are all periodic or eventually mapped into periodic cycles containing critical points, they do not have Levy cycles and are therefore Thurston equivalent to polynomials.   Let $\tilde{F}$ be the topological
gluing of $\tilde{f}$ and $\tilde{g}$.  By Theorem 2.1,   $\tilde{F}$ is realized by some rational map $\tilde{G}$ in $\mathcal{R}_{d_1+d_2 -d_0}$. By scaling conjugation if necessary, we may assume that the coefficient $\gamma = 1$ for $\tilde{G}$, see (\ref{fr}).   By construction, $\tilde{F}$ and $F$, and hence  $\tilde{G}$ and $F$,  share  the same puzzle structure  up to  depth $N$, which can be described as follows.

First note that each super-attracting cycle of $\tilde{F}$ is a holomorphic attracting cycle, i.e.,  $\tilde{F}$ is holomorphic in an open neighborhood of  each    such cycle \cite{ZJ}.  Hence  we can take homeomorphisms $\Phi_0, \:\Phi_1: S^2 \to \widehat{\Bbb C}$ which are isotopic to each other, so that
$$
\Phi_0 \circ \tilde{F} =  \tilde{G} \circ  \Phi_1,
$$
and moreover,  $\Phi_0$ and $\Phi_1$  are holomorphic and identified with each other   in an open neighborhood of each
 super-attracting cycle \cite{ZJ}.    For each pair of periodic rays in the initial graph $I$, which connect two super-attracting periodic points  of $\tilde{F}$ (either connect infinity and the origin, or connect infinity and some super-attracting periodic point, or connect the origin and some super-attracting periodic point),
 we may assume that $\Phi_0$ maps   the two rays  to  two periodic rays of $\tilde{G}$ which start from the
   two corresponding super-attracting periodic points  of $\tilde{G}$, and which land on the same repelling periodic point of $\tilde{G}$.  To see this, let $R$ denote the union of the two rays and $p$ be their period.  First let $\Phi_0(R)$ be a smooth arc.
 By lifting $\Phi_0$ through the following diagram,

$$
 \begin{CD}
          {S^2}         @  > \Phi_{i+1}     >  >    \widehat{\Bbb C}        \\
           @V  \tilde{F}  VV                         @VV \tilde{G} V\\
         {S^2}           @  >\Phi_i   >  >        \widehat{\Bbb C}
     \end{CD}
     $$
we obtain a sequence  of arcs,  $\{\Phi_{ip}(R)\}_{i \ge 0}$,  all of which
 connect two super-attracting periodic points of $\tilde{G}$,  and are homotopy to each other in $\widehat{\Bbb C} - P_{\tilde{G}}$  with the two end points being fixed.  Since all $\Phi_i$ are holomorphic and identified with each other in an open neighborhood of the super-attracting cycles,
each of the two end parts of  $\Phi_{ip}(R)$ is a part of  periodic
  ray in the immediate basin at this end point. As we lift the diagram, the two ray segments are extended longer and longer.   On the other hand,   since the inverse branch of $\tilde{G}$ decrease the orbifold metric and since $\Phi_0(R)$ is smooth,
  as  we lift the diagram, any middle part of $\Phi_0(R)$ shrinks to a point. In other words,
  $\Phi_{ip}(R)$ converges to
    the union of two periodic  rays which start  respectively from the two  end  points  and land on the same repelling periodic point. Let $\tilde{R}$ be the union od the two rays. Then $\tilde{R}$
   is in the same homotopy class of $\Phi_0(R)$.  For this reason, we may assume in the beginning that $\Phi_0(R) = \tilde{R}$.
    We may further assume that  $\Phi_0$ maps the equi-potentials in $I$ to the corresponding equi-potentials for $\tilde{G}$.
   In this way $\Phi_0$ maps the initial graph $I$ for $\tilde{F}$ to an initial graph for $\tilde{G}$.
Now by lifting the above diagram,  the maps $\{\Phi_i\}$  maps the puzzle system   for $\tilde{F}$   to the puzzle system for $\tilde{G}$,
    with the puzzle dynamics being preserved.

    It follows that  $F$ and $\tilde{G}$
    have the same puzzle structure up to depth $N$ since this holds for  $F$ and $\tilde{F}$.     Since all $\Phi_i$ are identified with each other in an open neighborhood of super-attracting cycles, the Boettcher coordinate for each super-attracting cycle for $\tilde{F}$ induces      the Boettcher coordinate for the corresponding super-attracting cycle for $\tilde{G}$.  We summarise these as follows.

 \begin{prop}\label{rs} For each $n \ge 1$, there is a rational map $G_n \in \mathcal{R}_{d_1+d_2 - d_0}$ with $\gamma = 1$
 such that $F$ and $G_n$  has the same puzzle structure up to depth $n$ in the following sense. \begin{itemize}  \item There exist an open neighborhood of each super-attracting cycle of $F$ which contains all the equi-potentials in $I$,   and a homeomorphism $\Phi_0$ of the sphere which is holomorphic and conjugates $F$ to $G_n$ on  these neighborhoods, \item $\Phi_0$ maps each ray in $I$ onto the corresponding ray of $G_n$ with the same external or internal angle, and moreover, $\Phi_0$ conjugates $F$ to $G_n$ on these rays, \item   $\Phi_0$  can be lift through
 $\Phi_i \circ F = G_n \circ \Phi_{i+1}$ until $\Phi_n$ so that $\Phi_i = \Phi_{i+1}$ on $F^{-i}(I)$ for all $0 \le i \le n-1$.

 \end{itemize}.

\end{prop}

\section{Compactness argument}

Let $\{G_n\}$  be the sequence in Proposition~\ref{rs}. We  may write
\begin{equation}\label{gg-f}
G_n(z) = z^m \frac{(z - \alpha_n^1)\cdots(z-\alpha_n^k)}{(z - \beta_n^1)\cdots(z-\beta_n^l)}
\end{equation}
with
  $m = d_2$, $l = d_2 - d_0$ and $k = d_1 - d_0$.
We need to prove  the following lemma.
\begin{lema}\label{cp}
There exist $0<r < R <\infty$ and $ \delta>0$ such that
\begin{enumerate}[label=\emph{(\arabic*)}]
    \item $r < |\alpha_n^i|, |\beta_n^j| < R$, and
    \item $|\alpha_n^i - \beta_n^j| > \delta$,
\end{enumerate}
  hold for all $n \ge 1$, $1 \le i \le k$ and $1 \le j \le l$.
\end{lema}

\begin{proof}

We  first show  that $|\alpha_n^i|$ and $ |\beta_n^j|$ have a finite upper bound. Suppose not.  Let
$$
M_n = \max_{1\le i\le k , 1 \le j\le l}\{|\alpha_n^i|, |\beta_n^j|\}.
$$
Then $ \limsup M_n =  \infty$. We may assume  $M_n \gg 1$. It is clear that one of the following annuli   contains  no  $\alpha_n^i$  or $\beta_n^j$,
$$
A_s = \{2^{-(s+1)}M_n < |z| < 2^{-s}M_n\},\: 0\le s \le k+l.
$$
Let $\gamma$ be the core curve of the annulus.  For  all $z \in \gamma$, we have
$$
|z| \asymp M_n, |z - \alpha_n^i| \asymp M_n, \hbox{ and } |z - \beta_n^j| \asymp M_n.
$$
Hence  $|G_n(z)| \asymp M_n^{d_1} \gg M_n$, so  $\gamma$  lies in  the basin of infinity. Let $\mathcal{J}_{G_n}$ be the Julia set of $G_n$.  Since  $G_n$ is post-critically finite, $\mathcal{J}_{G_n}$ is connected.

We now distinguish two cases.

  $\bold{Case \:1}$.   $\mathcal{J}_{G_n}$ and infinity  lie in the same component of $\widehat{\Bbb C} - \gamma$.   This implies that  the origin can not  see any point in  $\mathcal{J}_{G_n}$ via  rays in
  the immediate basin at the origin. This is impossible.

  $\bold{Case\: 2}$. $\mathcal{J}_{G_n}$  and the origin  lie in
   the same component of $\widehat{\Bbb C} - \gamma$.   Since all $\alpha_n^i$ and  $\beta_n^j$  are contained in the bounded components of $\widehat{\Bbb C} \setminus  \mathcal{J}_{G_n}$,   it follows  that all $\alpha_n^i$ and  $\beta_n^j$  lie in  the component of   $\widehat{\Bbb C} - \gamma$ which contains the origin.  On the other hand,
    the
point $\alpha_n^i$ or $\beta_n^j$ with modulus being $M_n$, must belong to the  component of $\widehat{\Bbb C} - \gamma$ which contains  infinity.    This is again a contradiction  since $\gamma$ separates  infinity and the origin.

Thus $|\alpha_n^i|$ and  $|\beta_n^j|$ are bounded above.
Next we show a positive lower bound.
Suppose not.  Let
$$
m_n = \min_{1\le i\le k , 1 \le j\le l}\{|\alpha_n^i|, |\beta_n^j|\}.
$$ Then $\liminf m_n = 0$. We may assume that $m_n \ll 1$. Consider the annuli
$$
\{z\:|\: 2^i m_n < |z| < 2^{i+1} m_n\}, 1 \le i \le l+1.
$$
At least one of  these annuli   contains   no $\beta_n^j, 1 \le j \le l$. Let $\gamma$ be the core curve of the annulus. Then for   $z \in \gamma$,
$$
|z| \asymp m_n, |z - \beta_n^j| \succeq m_n,  |z - \alpha_n^i|  = O(1).
$$
The third inequality follows from the fact that all the $|\alpha_n^i|$   have a finite upper bound. Since $m -l = d_0$, it follows that for all $z \in \gamma$,
$$
|G_n(z)| \preceq m_n^{d_0}.
$$ Hence  the forward iteration of    $G_n$ at $z$ converges to the origin. So $\gamma$ lies in  the attracting basin of the origin. But   $\gamma$ also separates  infinity from
 the point with modulus   $m_n$ (one of $\alpha_n^i$ or $\beta_n^j$),  and thus separates infinity from $\mathcal{J}(G_n)$ (This is because $\mathcal{J}(G_n)$ is connected, all $\alpha_n^i$ and  $\beta_n^j$  lie in   components of $\widehat{\Bbb C} - \mathcal{J}(G_n)$ not containing  infinity and the origin, and $\gamma$ does not intersect $\mathcal{J}(G_n)$). This implies that $\gamma$ lies in the component of $\widehat{\Bbb C} - \mathcal{J}(G_n)$ which contains infinity, and thus lies in the immediate basin of infinity.    This contradicts with the fact that it lies in  the attracting basin of the origin. Therefore  there is a positive lower bound for
 $|\alpha_n^i|$ and  $|\beta_n^j|$.

The proof of the second assertion  proceeds by contradiction as well.  Otherwise,  after passing to
 a subsequence if necessary, there exist at least one pair of a zero and a pole  with  the same limit.  For simplicity, suppose  there is exactly one such pair, say $\alpha_n^i \to z_0$ and $\beta_n^j\to z_0$. The general case  is similar.  Then there is a rational map $G$ of lower degree
such that $G_n \to G$ uniformly   outside  any neighborhood of $z_0$.  Let $w_0 = G(z_0)$.  Then for a small Jordan neighborhood $U$ of $w_0$, there is a small Jordan neighborhood of $z_0$, say $V_n$, such that $G_n(\partial V_n) = \partial U$, and $  G_n(V_n) = \widehat{\Bbb C}$.

Since $G_n$ is   post-critically  finite, there is a gluing curve $\Gamma_n$ which is invariant for $G_n$ \cite{ZX}.  Because  $V_n$ must contain the zero
 $\alpha_n^i $ and  the pole $\beta_n^j$,   we have $\Gamma_n \cap  \partial V_n \ne \emptyset$ and hence   $\Gamma_n \cap  \partial U \ne \emptyset$.
 By the first assertion,  $|G_n(z)| \asymp |z^{d_1}|$ for $z$ near infinity and $|G_n(z)| \asymp |z^{d_2}|$ for $z$ near the origin. It follows that $\mathcal{J}_{G_n}$ and thus
  $\Gamma_n$  are bounded away from zero and  infinity.
So when $U$ is small,  with respect to the harmonic measure at
the origin, except a small subset of $\Gamma_n$,   most   points in $\Gamma_n$ can be seen by the origin through geodesic rays   inside  $\Gamma_n$   that  do not  touch $\overline{U}$.   The same is true   from the viewpoint of infinity.
 Consequently,  we can select  a point   $x\in \Gamma_n$ that  is not a critical value, so that  the geodesic rays
 from both infinity and the origin arriving at $x$ do not touch $\overline{U}$.

 Since $G_n(V_n) = \widehat{\Bbb C}$, there  exists a point $y \in V_n$ with  $G_n(y)  = x$. We now have three possibilities.

 \begin{figure}[H]
\setlength{\unitlength}{1mm}
\begin{center}
\includegraphics[width=75mm]{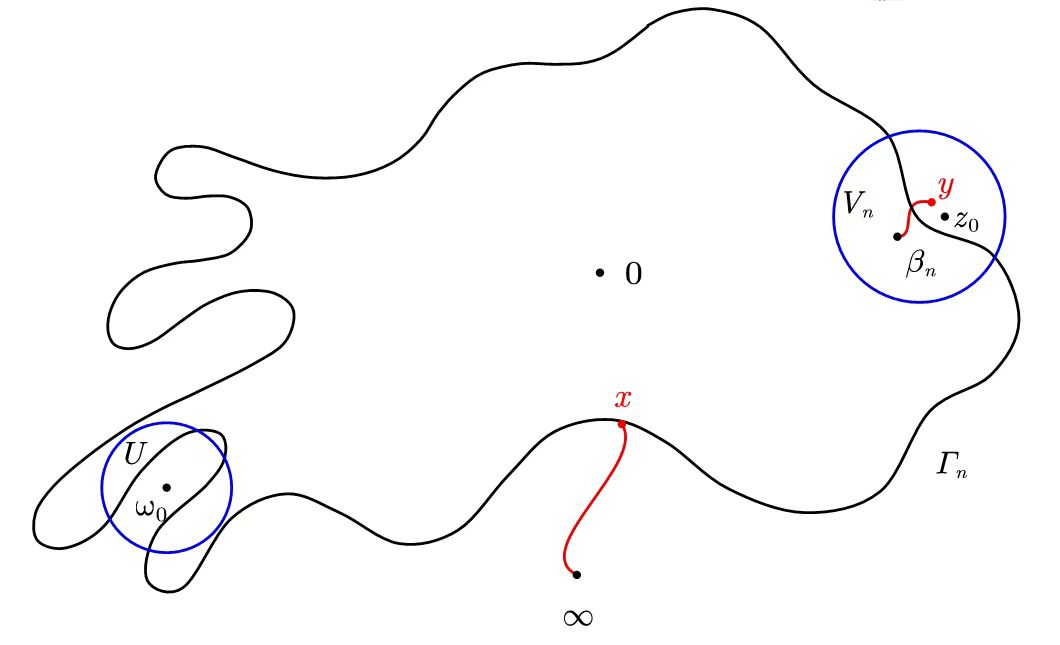}
\caption{Case 1. }
\label{Figure 1}
\end{center}
\end{figure}

\textbf{Case 1.}  The point $y$ lies  outside  $\Gamma_n$.   Let $R$ be  the geodesic ray  from infinity to the point $x$.   The
 inverse image of $R$  that contains $y$  is  a curve segment $S \subset V_n$  with $y$ as  one   end point.  The other end point of $S$  must  be a pole   inside  $\Gamma_n$.
 This implies that  an interior point of $S$  lies in $\Gamma_n$.
  Since $\Gamma_n$ is invariant, it follows that an interior point of $R$  lies in  $\Gamma_n$,  which is a contradiction. See Figure 12 for an illustration.

 \begin{figure}[H]
\setlength{\unitlength}{1mm}
\begin{center}

\includegraphics[width=75mm]{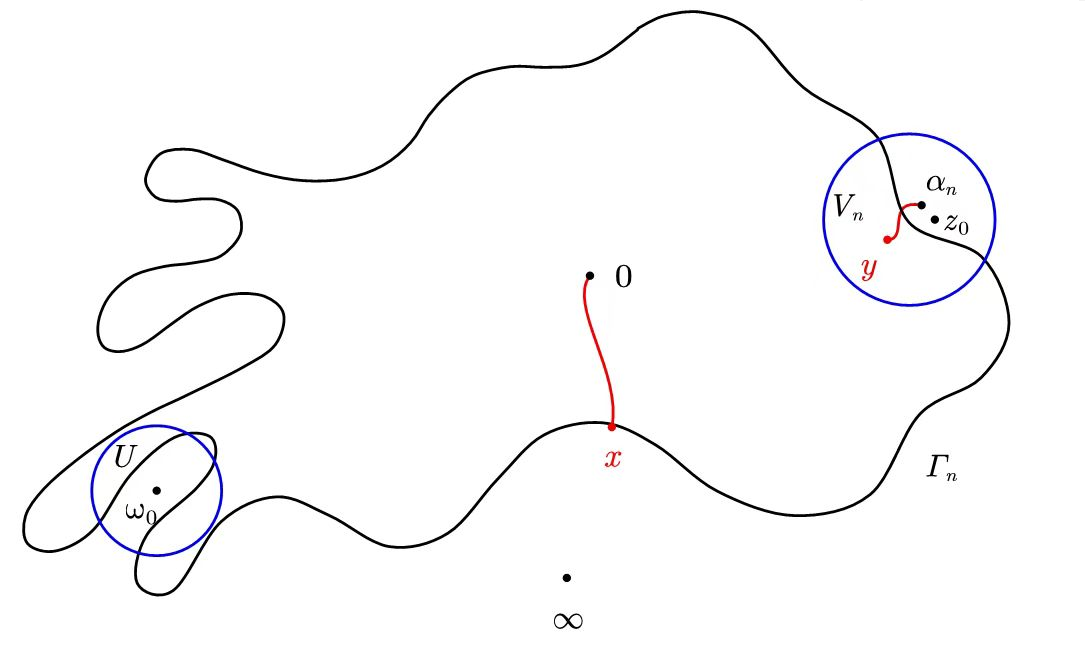}
\caption{Case 2. }
\label{Figure 1}

\end{center}
\end{figure}

\textbf{Case 2.}  The point $y$ lie   inside  $\Gamma_n$.  Considering the geodesic ray $R$  from the origin  to the point $x$.  A contradiction can be obtained  by a similar argument. See Figure 13 for an illustration.

 \begin{figure}[H]
\setlength{\unitlength}{1mm}
\begin{center}

\includegraphics[width=75mm]{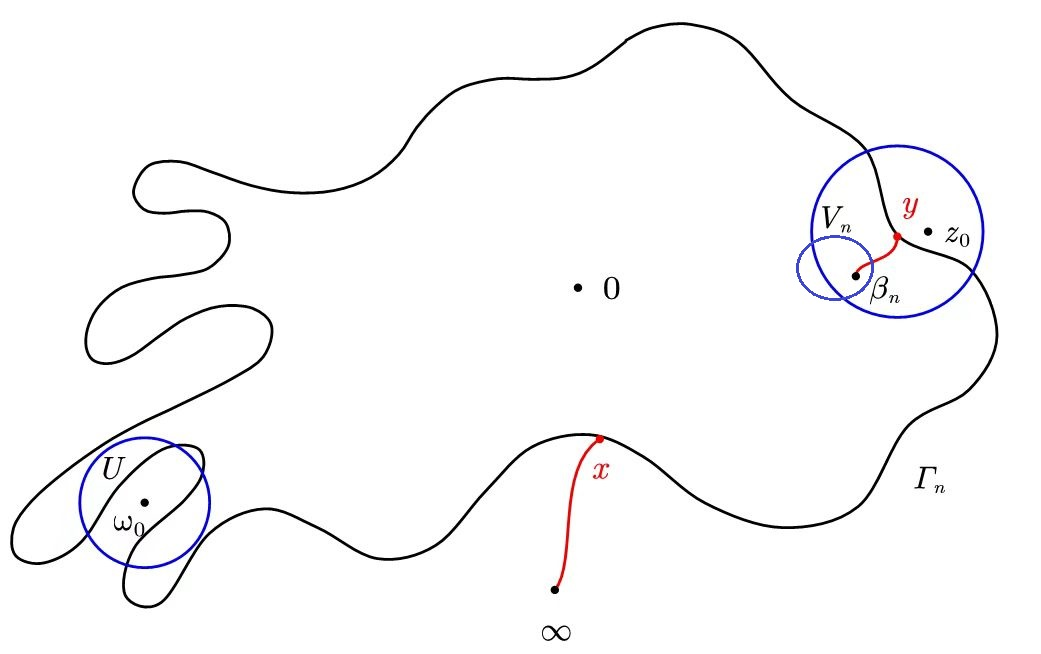}
\caption{Case 3. }
\label{Figure 1}

\end{center}
\end{figure}
\textbf{Case 3.}  The point $y$ lies on   $\Gamma_n$.  Take  the geodesic ray $R$  from  infinity to   $x$.   The
 inverse image of $R$  that contains  $y$  is  a curve segment $S \subset V_n$  with end points $y$ and a pole. Since $x$ is not a critical value, $y$ is not a critical point. So $S$ must pass  through some pre-image circle of $\Gamma_n$.
 This again implies that  an interior point of  $R$ lies on  $\Gamma_n$,  which  is a contradiction. See Figure 14 for an illustration.

 This proves a positive lower bound for the distance between $\alpha_n^i$ and $\beta_n^j$ and completes the proof.
  \end{proof}

By Lemma~\ref{cp} and by taking a subsequence if necessary, we may assume that $G_n$  converges uniformly  to some rational map $G \in \mathcal{R}_{d_1 + d_2 - d_0}$ with respect to the spherical metric.

Recall that  the   initial graph $I$ taken in (\ref{RE})  consists of periodic
rays and equipotential curves for $F$.  Since the rays are periodic, the
angles  for the rays must be rational.  By Proposition~\ref{rs},   for each ray in $I$, the corresponding ray for $G_n$  lands  at a repelling periodic point with the same external or internal angle.  But a priori, the ray with the same angle for the limit map $G$ may land at a parabolic periodic point.  To overcome this problem,  noting  that each parabolic cycle count at least one critical point,
we may take the set of periodic rays in $I$ to be large enough so that  a sub-graph of $I$ can be used to construct   puzzles, and moreover,  for this sub-graph, all the corresponding rays for $G$ land at repelling periodic points.  In particular, for any two rays for $F$ landing  at the same point, the two
corresponding rays for $G$ must   land at  the same point. Because otherwise, since $G_n$ is a perturbation of $G$ for all $n$ sufficiently large, the two rays for $G_n$ would also land at  two different points. This is impossible by  Proposition~\ref{rs}.
  Since $G_n$ converges uniformly to $G$,   by  Proposition~\ref{rs}, we have
  \begin{lema}\label{FGS}    $F$ and $G$ have the same puzzle structure in the following sense.
 \begin{itemize}  \item There exist an open neighborhood of each super-attracting cycle of $F$ which contains all the equi-potentials in $I$,   and a homeomorphism $\Phi_0$ of the sphere which is holomorphic and conjugates $F$ to $G$ on  these neighborhoods, \item $\Phi_0$ maps each ray in $I$ onto the corresponding ray of $G$ with the same external or internal angle, and moreover, $\Phi_0$ conjugates $F$ to $G$ on these rays, \item   $\Phi_0$  can be lift through $\Phi_i \circ F = G \circ \Phi_{i+1}$   so that $\Phi_i = \Phi_{i+1}$ on $F^{-i}(I)$ for all $i \ge 0$.
\end{itemize}
\end{lema}

\section{Non-renormalizable holomorphic dynamics}

\subsection{Uniform control of the puzzle size}
   Let $\mathcal{P}$ be the puzzle system of  $G$.   Since $F$ is non-normalizable,  it follows that $G$ is non-renormalizable.
  \begin{thm}\label{us}
For any $\epsilon > 0$, there is an $M \ge 1$
such that for all $n \ge M$ and  every $P \in \mathcal{P}$ of depth $n$,  we have
$$
{\rm diam}(P) < \epsilon.
$$
\end{thm}

Theorem~\ref{us} is  a consequence of     ideas in non-renormalizable holomorphic dynamics, which have been well developed
since the fundamental work in \cite{KL1}\cite{KL2}\cite{KSS}.  The reader may refer to
 \cite{KS}\cite{QY}\cite{RY}\cite{RYZ}\cite{YZ} for   subsequent work on this aspect. Since Theorem~\ref{us} is the key in the next step of our proof, we would like to outline the proof as follows. Before that, let us introduce some notions first.

  We say a critical point $c$ of $G$ is $\emph{recurrent}$ if for any $n \ge 1$, there is a $k \ge 1$ such that $G^{k}(c) \in P_n(c)$.  Otherwise, we say it is $\emph{non-recurrent}$.  We say   $c$  is $\emph{reluctantly recurrent}$  if there exist   $k_n  \to \infty$, an integer $D>0$ and a depth-$0$ puzzle $P$
  such that
$$
G^{k_n}: P_{k_n}(c) \to P
$$
has degree less than $D$.  Otherwise, we say it is $\emph{persistently recurrent}$. We say $c$ is $\emph{periodically persistently recurrent}$ if it is renormalizable.

For $x \ne y$, we say $y \in \omega(x)$ if for any $n \ge 1$, there is some $m \ge 1$ such that $G^m(x) \in P_n(y)$. We say  $x$ is equivalent to $y$, denoted as $x \sim y$,  if $x \in \omega(y)$ and $y \in \omega(x)$.

\begin{proof}
The proof is by contradiction. Suppose it were not true. Then we have a sequence of nested puzzles $P_n$  such that  \begin{equation}\label {ac}
{\rm diam}(P_n) > \epsilon.\end{equation}   Since $\overline{P_n}$ is a sequence of nested compact set, $\bigcap \overline{P_n} \ne \emptyset$.  Take $z \in  \bigcap \overline{P_n}$.
  We claim $z \in \bigcap  {P_n}$. Since otherwise, there exists an integer $k$ such that $\zeta = G^k(z)$ is the landing point of a periodic ray in the initial graph $I$. Then $P_{n-k}$ is a sequence of nested puzzles with $\zeta$ being a common boundary point. Since $G$ is non-renormalizable and $\zeta$ is a repelling periodic point, it follows that ${\rm diam}(P_{n-k}) \to 0$ and thus ${\rm diam}(P_{n}) \to 0$. This is a contradiction. The claim has been proved.  So $P_n = P_n(z)$ and by  (\ref{ac}) we have
  $$
  {\rm diam}(P_n(z)) > \epsilon.
  $$

For a non-periodically
persistently recurrent critical point $c$,  if there exists,  according to \cite{KL2}\cite{KSS},
    there is a sequence of pais of puzzles containing $c$, say $(K_n, K_n')$ and $\mu > 0$, such that
\begin{itemize}
\item $\rm{mod}(K_n'\setminus K_n) > \mu$,
\item $K_n'\setminus K_n$ does not intersect the forward orbits of the critical points in the equivalent class of $c$.
\end{itemize}

For a  reluctantly recurrent critical point $c$,  if there exists,  there exist a $D > 0$,
 a sequence $k_n \to \infty$ and a depth-$0$ puzzle $P$    such that the degree of
 $$
 G^{k_n}: P_{k_n}(c) \to P
 $$ is bounded by $D$.   By the assumption  that  the rays in $I$ does not land on the forward orbit of $c$,
   we can always find a puzzle $Q$ containing $G^{k_n}(c)$ and  compactly contained
   in $P$.   We  can then pull back the annulus $P \setminus \overline{Q}$
  by $G^{n_k}$ to get a sequence of disjoint   annuli  around $c$.

The argument is now divided into two cases.

$\bold{Case \: I.}$  There exists some $c$ such that $c \in \omega(z)$.

If $c$ is non-periodically
persistently recurrent,   we have a sequence $k_n \to \infty$ so that $G^{k_n}(z)$ first enters $K_n$.
If there are infinitely many $K_n' \setminus K_n$ which  intersect a critical orbit of some $c'$ which is not equivalent to $c$, by  the 3th assertion
 of Lemma 5.9 in \cite{RY}, $z$ has the ``bounded degree property''.  We can thus use the same argument as in  the case of reluctantly recurrent critical point to get a sequence of disjoint annuli around $z$. Otherwise, there are infinitely many $K_n' \setminus K_n$ which  do not intersect any critical orbit.
Then the pull back of the annuli $K_n'\setminus K_n$ by $G^{-k_n}$ produces a sequence of annuli  $Q_n'\setminus Q_n$  around $z$
with ${\rm mod}(Q_n'\setminus Q_n) \ge \mu/C$ for some uniform $C >1$. A contradiction can be derived  from the
Gr\"{o}tzsch inequality.

If $c$ is reluctantly  recurrent,  as we said before we have a sequence of nested puzzle neighborhoods $P_{k_n}(c)$  and  a puzzle $P$ such that the degree of
$
G^{k_n}: P_{k_n}(c) \to  P
$ is bounded by some uniform $D> 0$. Let $l_n \ge 0$ be the least  integer  such that
$$
G^{l_n}(z) \in P_{k_n}(c).
$$
Then  the degree of

$$
G^{k_n +l_n}: P_{l_n+k_n}(z) \to  P
$$ is  uniformly bounded. By the assumption  that  the rays in $I$ does not land on the forward orbit of $z$,
   we can always find a puzzle $Q$ containing $G^{k_n+ l_n}(z)$ and  compactly contained
   in $P$.  We can then get a sequence of disjoint annuli around $z$  by pulling back the annulus $P \setminus \overline{Q}$.  A contradiction is again derived   from the
Gr\"{o}tzsch inequality.

$\bold{Case \: II.}$  $\omega(z)$ contains  none critical points.  Then there is an $n_0$ such that for any $n > n_0$, the map
$$
G^{n-n_0}: P_n(z) \to P_{n_0}(G^{n-n_0}(z))
$$
is a holomorphic isomorphism.   The    same argument as above can be used to get a contradiction.

\end{proof}

\subsection{Rigidity for $G$} The quasiconformal rigidity for
 non-renormalizble holomorphic rational maps  has been known    since the work of \cite{KL2}\cite{KSS}. For completeness, let us sketch the outline of the proof as follows.

Suppose there are two rational maps $G$ and $G'$  both of which realize $F$.  For $k \ge 0$, let $\mathcal{P}_k$ and  $\mathcal{P}_k'$ denote  the union  of
the collection of puzzles of level $k$ for $G$ and $G'$ respectively.     Then we may define $h$ which maps the outside of
$\mathcal{P}_0$ conformally onto the outside of $\mathcal{P}_0'$
so that
$$
h \circ G = G' \circ h.
$$
Since $G$ is combinatorially equivalent to $G'$,  by lifting the above equation  we can   conformally extend $h$  to  a  holomorphic isomorphism between the Fatou set of $G$ and $G'$.  Since points in the Julia sets of $G$ and $G'$ can be represented as nested puzzle sequence by Theorem~\ref{us}, and since $h$ gives a one-to-one correspondence between puzzles of $G$ and $G'$, we can extend $h$ to a homeomorphism of the sphere to itself.    Next we need to show that such $h$ must be   quasi-conformal.

\begin{lema}
[QC-CRITERION, \cite{KSS}] Let $H\geq1,\ M>1$ and $m>0$ be constants, and $h: U\to \tilde{U}$ be an orientation-preserving homeomorphism between two Jordan domains. Let $X_0$ be a subset of $U$ with $\text{mes}(X_0)=\text{mes}(h(X_0))=0$ and $X_1$ be a subset of $U\setminus X_0$. Let
$$
\underline{H} (h, x) = \liminf_{r \to 0} \frac{{\rm sup}_{|y-x|=r}|h(y)-h(x)|}{{\rm inf}_{|y-x|=r}|h(y)-h(x)|}\in[1,\infty].
$$ Assume that the following hold
\begin{enumerate}[label=\emph{(\arabic*)}]
\item for any $x\in U\backslash (X_0\cup X_1),\ \underline{H}(h, x)< H;$
\item for any $x\in X_0,\ \underline{H} (h, x)<\infty;$
\item there exists a family $\mathcal{V}$ of pairwise disjoint topological disks which form a covering of $X_1$, such that for each $P\in \mathcal{V}$, we have
\begin{enumerate}[label=\emph{(\alph*)}]
\item both $P$ and $h(P)$ have M-Shape;
\item $\text{mod}(U\backslash \overline{P})>m$ and $\text{mod}(\tilde{U}\backslash \overline{h(P)})>m.$
\end{enumerate}
\end{enumerate}
Then there exists a K-quasiconformal mapping $\tilde{h}:U \to \tilde{U}$ such that $$h \vert_{\partial U}=\tilde{h} \vert_{\partial U},$$ where $K > 1$ is a constant depending only on $H,\ M \ \text{and}\ m$.
\end{lema}

The reader may refer to   \cite{KSS}\cite{RYZ}\cite{Shen}\cite{YZ} for the details of the remaining argument. Let us just sketch the basic idea as follows. For $k \ge 0$,  $h$ maps   each puzzle of $G$ with depth $k$ to the corresponding puzzle of $G'$. Note that our puzzles are not Jordan domains. But this does not matter since
$h$ can be lifted  to a homeomorphism between two closed unit disks  through Riemann   holomorphic isomorphisms.
Apply the lemma to each such puzzle.
  The union of $X_0$ and $X_1$ is the Julia set of $G$ contained in the puzzle.  Since $h$ is conformal in the Fatou set, the condition (1) is satisfied. The set $X_0$ consists of good points in the Julia set,
which satisfy the bounded degree condition and therefore, the distortion is bounded
in a sequence of puzzle neighborhoods shrinking to the points. It follows that $X_0$ is a null set.  Thus the condition (2) is satisfied.  The set $X_1$ consists of the remaining points in the Julia set. The required annulus condition (b) of (3) is guaranteed by  enhanced nested puzzle sequence \cite{KSS}, which implies the bounded distortion of the pull back of the puzzles in the enhanced nest.  This, together with the power law implies the bounded shape condition (a) of (3).  We can thus  extend $h$ to all puzzles of depth $k$ and
get a $K$-quasiconformal homeomorphism $h_k$ of the sphere. Let $k \to \infty$, we get a $K$-quasiconformal conjugation of $G$ and $G'$.

Let us show $h$ is actually conformal.  Suppose it  were  not.   Then  there would be an invariant line field   supported on a subset of $X_1$ with positive Lebesgure measure. This implies that some persistently recurrent critical point of $G$ must be a Lebesgue point of the line filed.  We then get  a contradiction by the recurrence of the critical point  since  the pull back to the critical point produces diverse directions of the lines  \cite{Shen}.

\section{The Proof of the Main Theorem}

\subsection{The realization of the geometric gluing}

By the rigidity assertion in  $\S 6.2$,  the sequence in Proposition~\ref{rs} has a unique limit  $G$.
In this section  we will
 show that $G$ realizes the simultaneous uniformization of $f$ and $g$ and is thus the geometric gluing of $f$ and $g$.  By definition, it suffices to show that there are two  holomorphic   conjugations
  $$\phi:  {B_\infty(f)} \to   {B_\infty(G)},     \:\:\: \psi:   {B_\infty(g)} \to  {B_0(G)}$$   which can be respectively  extended to  conjugations   between the boundaries, and that the two polynomial Julia sets are glued together along a Jordan curve.
We do this only for $\phi$, and the same argument applies to    $\psi$  as well.
$$
 \begin{CD}
         B_\infty(f)          @  > \phi      >  >    {B_\infty(G)}        \\
           @V f  VV                         @VV  G  V\\
          B_\infty(f)           @  >\phi   >  >        {B_\infty(G)}
     \end{CD}
     $$

Let $\Phi_0$ be the homeomorphism guaranteed by Lemma~\ref{FGS}.  Recall that $\Phi_0$ is holomorphic and conjugates $F$ to $G$
in an open neighborhood of  super-attracting cycles of $F$.  According to Lemma~\ref{FGS} we can lift $\Phi_0$ through the following diagram for all $i \ge 0$,
$$
 \begin{CD}
         S^2         @  > \Phi_{i+1}    >  >    \widehat{\Bbb C}        \\
           @V F  VV                         @VV  G  V\\
         S^2          @  >\Phi_i   >  >        \widehat{\Bbb C}
     \end{CD}
 $$
Since $F = f$ on $B_\infty(f)$, this yields a holomorphic isomorphism from $B_\infty(f)$ to $B_\infty(G)$ which conjugates
  $f$ and $G$. Let us denote it as $\phi$.

   Recall that  the initial graph of $F$ induces an initial graph of $f$ (see $\S 3$, immediately after Lemma~\ref{pp}), and thus induces a puzzle system for $f$.
  We  need to note that, for any $n \ge 1$, the depth-$n$
     puzzle neighborhoods of points in $\mathcal{J}_f$ for $f$ and $F$
      are different,  but the part of the boundary of each of the two puzzles, which can be seen from infinity, is the same.
       \begin{figure}[H]
\setlength{\unitlength}{1mm}
\begin{center}

\includegraphics[width=120mm]{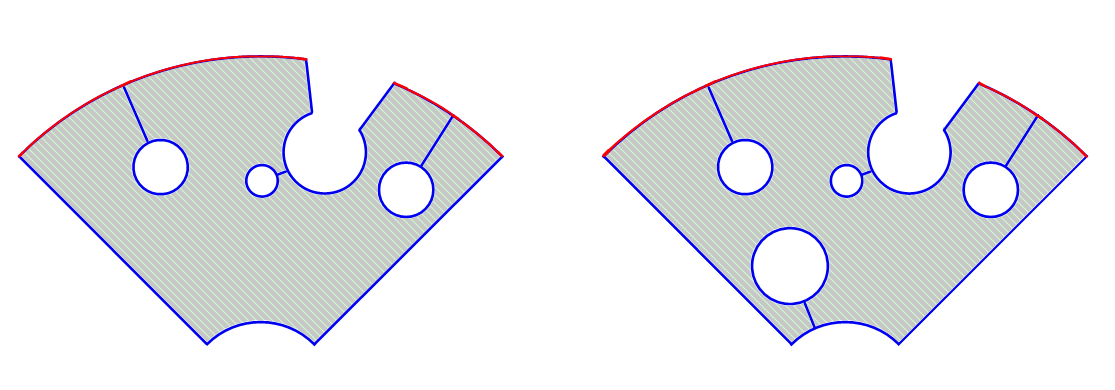}
\caption{ The left hand is the puzzle for $f$ and the right hand is the puzzle for $F$. The red curve segment
 of the boundary is the part which can be seen from infinity through the rays}
\label{Figure 1}

\end{center}
\end{figure}

      Since $F$ and $G$ have the same puzzle structure,
     the map $\phi$ provides a correspondence between
      puzzles $P$ for $f$ to   puzzles $Q$ for $G$ $-$ $P$ can be seen by  a ray $R$
      in $B_\infty(f)$  if and only if $Q$ can be seen by the ray $\phi(R)$ in $B_\infty(G)$.  Let us denote this correspondence by $Q = \phi(P)$.

       Since $f$ is non-renormalizable, the nested  puzzles  for  any point $z \in \mathcal{J}_f$
      shrinks  to $z$.   So we may regard a point  in the Julia set as a nested sequence of  puzzles.  This makes sense even when the point is a common boundary point of a nested sequence of puzzles.  By  Theorem~\ref{us},   the
       nested sequence of the corresponding
         puzzles for $G$  shrink to a point $ w \in  \partial B_{\infty}(G)$. By setting  $\phi(z) = w$ we extend $\phi$ to a homeomorphism
             $\phi: \mathcal{J}_f \to \partial B_{\infty}(G)$  which conjugates the dynamics of $f$ and $G$.  The same idea can be used to construct $\psi$.

    Since $\phi: \mathcal{J}_f \to \partial B_{\infty}(G)$ is a homeomorphism and since $\Gamma_f$ is a Jordan curve,
     is follows  that $\phi(\Gamma_f) \subset  \partial B_{\infty}(G)$ is a Jordan curve.  By construction, every puzzle neighborhood of points in $\phi(\Gamma_f)$ intersects $B_0(G)$. So
       the points in  $\phi(\Gamma_f)$
       can be seen from the origin through the rays in $B_0(G)$, that is,  $\phi(\Gamma_f) \subset  \partial B_{0}(G)$.  All these implies
       $$
       \phi(\Gamma_f) \subset  \partial B_{\infty}(G) \cap  \partial B_{0}(G).
       $$ To see the above relation must be an equation, we first note that a point belongs to $\partial B_{\infty}(G) \cap  \partial B_{0}(G)$ if and only if it belongs to a nested sequence of puzzles all of which can be seen from both infinity and the origin.  By construction, a
        puzzle for $F$ can be seen from both  infinity and the origin if and only if the puzzle  intersects the gluing curve for $F$.  Since a puzzle for $F$ intersects the gluing curve if and only if the corresponding puzzle for $f$ intersects $\Gamma_f$, and
        Since $G$ and $F$ have the same puzzle structure, a puzzle $Q$
        for $G$ can be seen from both  infinity and the origin if and only if $Q = \phi(P)$ where $P$ is a puzzle for $f$  intersecting $\Gamma_f$.       So we must have
      $$
      \phi(\Gamma_f)  =   \partial B_{\infty}(G) \cap  \partial B_{0}(G).
      $$
      In the same way we obtain
      $$
      \psi(\Gamma_g)  =  \partial B_{\infty}(G) \cap    \partial B_{0}(G).
      $$ Thus
      $
      \phi(\Gamma_f)  = \psi(\Gamma_g) =\partial  B_{\infty}(G) \cap  \partial B_{0}(G)
      $ is a Jordan curve along which the two polynomial Julia sets  are glued into $\mathcal{J}_G$.  This implies that  $G$ is the geometric gluing of $f$ and $g$.

\subsection{The continuity  of  $\mathcal{G}$}
To prove  the continuity, assume
 $(f, g), (f_n, g_n) \in \mathcal{N}_{d_1, d_0} \times \mathcal{N}_{d_2, d_0}$ with $$(f_n, g_n)
 \to (f, g).$$   It suffices to prove that $\mathcal{G}(f_n, g_n) \to \mathcal{G}(f,g)$.  We
  may write $\mathcal{G}(f_n, g_n)$ in the form of (\ref{gg-f}).
   Noting that the proof of Lemma~\ref{cp}   only utilizes the fact that the Julia set of $G_n$ is connected, which is obviously
   true for $\mathcal{G}(f_n, g_n)$. So the same argument can be used to prove that $\mathcal{G}(f_n, g_n)$ lies in a compact family. But on the other hand,  any limit of the sequence $\mathcal{G}(f_n, g_n)$ must have the same puzzle structure as that of $\mathcal{G}(f, g)$, and thus must be equal to $\mathcal{G}(f, g)$ by the rigidity assertion in $\S6$.  This implies that $\mathcal{G}(f_n, g_n) \to \mathcal{G}(f,g)$ as $n \to \infty$. The continuity of $\mathcal{G}$ follows.



\begin{thebibliography}{99}


\bibitem{BE} L. Bers, {\em Simultaneous unifomrmization}, Bull. Amer. Math. Soc., \textbf{66} (1960), 94$-$97.

\bibitem{B} B. Branner, {\em Puzzles and para-puzzles of quadratic and cubic polynomials}, Proc. Sympos. Appl. Math., \textbf{49} (1994), 31$-$69.

\bibitem{BEKM} X. Buff, A. L. Epstein, S. Koch, D. Meyer, K. Pilgrim, M. Rees and T. Lei, {\em Questions about polynomial matings}, Ann. Fac. Sci. Toulouse Math., (6) \textbf{21} (2012), no. 5, 1149$-$1176.

\bibitem{BH} B. Branner and J. H. Hubbard, {\em The iteration of cubic polynomials, part II: patterns and parapatterns}, Acta Math., \textbf{169} (1992), 229$-$325.

\bibitem{CDK} T. Clark, K. Drach, O. Kozlovski and et al, {\em The dynamics of complex box mappings}, Arnold Math. J., \textbf{8} (2022), no. 2, 319$-$410.




\bibitem{DH1} A. Douady and J. Hubbard, {\em On the dynamics of polynomial-like mappings}, Ann. Sci. ¨¦cole Norm. Sup., (4) \textbf{18} (1985), no. 2, 287$-$343.

\bibitem{DH} A. Douady and J. Hubbard, {\em A Proof of Thurston's Topological Characterization of Rational Functions}, Acta Math., \textbf{171} (1993), 263$-$297.


\bibitem{Do} A. Douady, {\em Syst\`{e}mes dynamiques holomorphes}, Bourbaki seminar, Vol. 1982/83, Ast¨¦risque, vol. 105, Soc. Math. France, Paris, 1983, 39$-$63.

\bibitem{D1} K. Drach,  {\em Rigidity of rational maps: an axiomatic approach, manuscript in preparation.}

\bibitem{D2} K. Drach, Y. Mikulich, J. R\"{u}ckert and D. Schleicher, {\em A combinatorial classification of postcritically
fixed Newton maps}, Ergod. Theor. Dyn. Syst., \textbf{39} (2019), no. 11, 2983$-$3014.

\bibitem{E} A. L. Epstein, {\em Counterexamples to the quadratic mating conjecture}, Manuscript, 1998.

\bibitem{EM} A. L. Epstein and M. Yampolsky, {\em Geography of the cubic connectedness locus: intertwining surgery}, Ann. Sci. \'{E}cole Norm. Sup., (4) \textbf{32} (1999), no. 2, 151$-$185.


\bibitem{H} J. H. Hubbard,
{\em Local connectivity of Julia sets and bifurcation loci: three theorems of J.-C. Yoccoz}, In:  Topological methods in modern mathematics, (1991), 467$-$511.




\bibitem{H1}  P. Ha\"{i}ssinsky, {\em Parabolic surgery}, C. R. Math. Acad. Sci. Paris, \textbf{327} (1998), 195$-$198.




\bibitem{IK} H. Inou and J. Kiwi, {\em Combinatorics and topology of straightening maps, I: Compactness and bijectivity},
Adv. Math., \textbf{231} (2012), no. 5, 2666$-$2733.

\bibitem{In} H. Inou, {\em Combinatorics and topology of straightening maps II: Discontinuity}, arXiv:0903.4289, 2018.




\bibitem{KL1}  J. Kahn and M. Lyubich, {\em The quasi-additivity law in conformal geometry}, Ann. of Math., (2) \textbf{169} (2009), no. 2, 561$-$593.

\bibitem{KL2}  J. Kahn and M. Lyubich, {\em Local connectivity of Julia sets for unicritical polynomials}, Ann. of Math., (2) \textbf{170} (2009), no. 1, 413$-$426.

\bibitem{KSS} O. Kozlovski, W. Shen and S. van Strien, {\em Rigidity for real polynomials}, Ann. of Math., (2) \textbf{165} (2007), no. 3, 749$-$841.


\bibitem{KS} O. Kozlovski and  S. van Strien, {\em Local connectivity and quasi-conformal rigidity of non-renormalizable polynomials}, Proc. Lond. Math. Soc., (3) \textbf{99} (2009), no. 2,  275$-$296.

\bibitem{L} M. Lyubich,{\em Dynamics of quadratic polynomials. I, II}, Acta Math., \textbf{178} (1997), no. 2, 185$-$247, 247$-$297.


\bibitem{McM} C. McMullen,  {\em Simultaneous uniformazation of Blaschke products}, unpublished manuscripts.
 6.


\bibitem{Mil} J. Milnor, {\em Local connectivity of Julia sets: expository lectures}, in The Mandelbrot Set, Theme and Variations, London Mathematical Society Lecture Note Series, Vol. 274, Cambridge University Press, Cambridge, 2000, 67$-$116.


\bibitem{P1} C. L. Petersen,
{\em Local connectivity of some Julia sets containing a circle with an irrational rotation},
Acta Math., \textbf{177} (1996), no. 2, 163$-$224.

\bibitem{P2} C. L. Petersen,
{\em Puzzles and Siegel disks}, In: Progress in Complex Dynamics (H. Kriete, ed.), Pitman Res. Notes Math. Ser., 387, Longman, Harlow, 1998.



\bibitem{PR} C. L. Petersen and P. Roesch, {\em The parabolic Mandelbrot set}, arXiv: 2107.09407, 2021.

\bibitem{PZ} C. L. Petersen and S. Zakeri,
{\em On the Julia set of a typical quadratic polynomial with a Siegel disk},
Ann. of Math., (2) \textbf{159} (2004), no. 1, 1$-$52.

\bibitem{QWY} W. Qiu, X. Wang and Y. Yin, {\em Dynamics of McMullen maps}, Adv. Math., \textbf{229} (2012), no. 4, 2525$-$2577.


\bibitem{QY} W. Qiu and Y. Yin, {\em Proof of the Branner-Hubbard conjecture on Cantor Julia sets}, Sci. China Ser. A, \textbf{52} (2009), no. 1, 45$-$65.

\bibitem{Re} M. Rees, {\em Realization of matings of polynomials as rational maps of degree two}, Manuscript, 1996.

\bibitem{R3} P. Roesch, {\em
Topologie locale des m\'ethodes de Newton cubiques: plan dynamique}
C. R. Acad. Sci. Paris S\'er. I Math., \textbf{326} (1998), no. 10, 1221$-$1226.

\bibitem{R1} P. Roesch, {\em Hyperbolic components of polynomials with a fixed critical point of maximal order}, Ann. Sci. \'{E}cole Norm. Sup., (4) \textbf{40}  (2007), no. 6, 901$-$949.

\bibitem{R2} P. Roesch, {\em On local connectivity for the Julia set of rational maps: Newton's famous example}, Ann. of Math., (2) \textbf{168} (2008), no. 1, 127$-$174.




\bibitem{RY} P. Roesch and Y. Yin, {\em Bounded critical Fatou components are Jordan domains for polynomials}, Sci. China Math., \textbf{65} (2022), no. 2, 331$-$358.

\bibitem{RYZ} P. Roesch, Y. Yin and Z. Zeng, {\em Rigidity of non-renormalizable Newton maps}, Sci. China Math., \textbf{67} (2024), no. 4, 855$-$872.


\bibitem{Shen} W. Shen, {\em On the measurable dynamics of real rational functions}, Ergodic Theory Dynam. Systems, \textbf{23} (2003), no. 3, 957$-$983.

\bibitem{ShT1} M. Shishikura and L. Tan, {\em A family of cubic rational maps and matings of cubic polynomials}, Experiment. Math, \textbf{9} (2000), no. 1, 29$-$53.

\bibitem{SW} W. Shen and Y. Wang, {\em Primitive tuning via quasiconformal surgery}, Israel J. Math., \textbf{245} (2021), no. 1, 259$-$293.

\bibitem{Tan} L. Tan, {\em Matings of quadratic polynomials}, Erg. Th. and Dyn. Sys., \textbf{12} (1992) 589$-$620.

\bibitem{WYZ} X. Wang, Y. Yin and J. Zeng, {\em Dynamics of Newton maps}, Ergodic Theory Dynam. Systems, \textbf{43} (2023), no. 3, 1035$-$1080.


\bibitem{Yo} J. C. Yoccoz, {\em On the local connectivity of the Manderbrot set}, unpublished.

\bibitem{YaZa} M. Yampolsky and S. Zakeri, {\em Mating Siegel quadratic polynomials}, J. Amer. Math. Soc., \textbf{14} (2001), no. 1, 25$-$78.

\bibitem{YZ} Y. Yin and Y. Zhai, {\em No invariant line fields on Cantor Julia sets}, Forum Math., \textbf{22} (2010), no. 1, 75$-$94.

\bibitem{ZJ} G. Zhang and Y. Jiang, {\em Combinatorial characterization of sub-hyperbolic rational maps}, Adv. Math.  \textbf{221} (2009), 1990$-$2018.

\bibitem{Zhang} G. Zhang, {\em Jordan mating is always possible for polynomials}, Math. Z., \textbf{306} (2024), no. 4, 10 pp.

\bibitem{ZX} X. Zhang, {\em Invariant Jordan curves in the Julia sets of rational maps},
AIMS Math., \textbf{10} (2025), no. 7, 16664$-$16675.



\end{thebibliography}
\end{document}